\documentclass[11pt]{article}
\usepackage[utf8]{inputenc}
\usepackage[T1]{fontenc}
\usepackage{graphicx}
\usepackage{longtable}
\usepackage{wrapfig}
\usepackage{rotating}
\usepackage[normalem]{ulem}
\usepackage{amsmath}
\usepackage{amssymb}
\usepackage{capt-of}
\usepackage{hyperref}
\usepackage[margin=1in]{geometry}
\usepackage{amsthm,thmtools}
\usepackage{stmaryrd}
\usepackage{authblk}
\usepackage{float}
\newtheorem{theorem}{Theorem}
\newtheorem{proposition}{Proposition}
\newtheorem{definition}{Definition}
\newtheorem{lemma}{Lemma}
\newcommand{\RR}{\mathbb R}
\newcommand{\N}{\mathbb N}
\newcommand{\proba}{\mathbb P}
\newcommand{\expect}{\mathbb E}
\newcommand{\tr}{\textsf{T}}
\author{Antonin Della Noce\thanks{Corresponding author: \texttt{antonin.della-noce@centralesupelec}}}
\author{Paul-Henry Cournède}
\affil{Université Paris-Saclay, CentraleSupélec, Lab of Mathematics and Informatics (MICS), 9 rue Joliot Curie, 91192, Gif-sur-Yvette, France}
\date{\today}
\title{Analysis of the dynamics induced by a competition index in a heterogeneous population of plants: from an individual-based model to a macroscopic model}
\hypersetup{
 pdfauthor={Antonin Della Noce, Paul-Henry Cournède},
 pdftitle={Analysis of the dynamics induced by a competition index in a heterogeneous population of plants: from an individual-based model to a macroscopic model},
 pdfkeywords={},
 pdfsubject={},
 pdfcreator={Emacs 26.3 (Org mode 9.5.1)}, 
 pdflang={English}}
\begin{document}

\maketitle
\tableofcontents

\begin{abstract}
\label{sec:orga17ff1b}

Competition indices are models frequently used in ecology to account for the impact of density and resource distribution on the growth of a plant population. They allow to define simple individual-based models, by integrating information relatively easy to collect at the population scale, which are generalized to a macroscopic scale by mean-field limit arguments. Nevertheless, up to our knowledge, few works have studied under which conditions on the competition index or on the initial configuration of the population the passage from the individual scale to the population scale is mathematically guaranteed. We consider in this paper a competition index commonly used in the literature, expressed as an average over the population of a pairwise potential depending on a measure of plants' sizes and their respective distances. In line with the literature on mixed-effect models, the population is assumed to be heterogeneous, with inter-individual variability of growth parameters. Sufficient conditions on the initial configuration are given so that the population dynamics, taking the form of a system of non-linear differential equations, is well defined. The mean-field distribution associated with an infinitely crowded population is then characterized by the characteristic flow, and the convergence towards this distribution for an increasing population size is also proved. The dynamics of the heterogeneous population is illustrated by numerical simulations, using a Lagrangian scheme to visualize the mean-field dynamics.
\end{abstract}

\section{Introduction}
\label{sec:orge597577}

Individual-based models (IBM) are powerful tools to explain the macroscopic
behavior of a complex system. In a population with interacting individuals,
taking into account the multi-scale effect is essential to understand the
eventual steady regime and the overall spatial patterns formed by their
evolution.

Regarding the modeling of plant populations, the Bolker-Pacala-Dieckmann-Law
(BPDL) model accounts for the processes of birth, death, and competition for
resources between individual plants [\ref{org679f66f},\ref{orga5cb81b}] at a macroscopic scale. A
rigorous formulation of the dynamics of this model as a point process was
obtained by Fournier and Méléard [\ref{org9af7581}]. Under specific configurations,
the dynamics of a large population can be approximated by a deterministic
process called the mean-field limit of the microscopic process. This object is a
formal representation of the macroscopic level of the population, in the sense
that at this scale, the individuals constituting the population are
indistinguishable or are not clearly defined.

Initially, the plants in the BPDL model were only described by their positions
in the plane, and the evolution of the corresponding point process was piecewise constant. The competition between the plants was then added to the natural
mortality rate, i.e., the frequency at which points disappeared from the
collection. This purely discrete-time model was then refined by Campillo and
Joannides [\ref{org960cc74}], incorporating a continuous growth process of
individuals. This representation allowed the effect of competition on plant
morphology to be accounted for in conjunction with their spatial distribution
[\ref{org9115032}, \ref{org1ddeaac}]. In most of these works, the population is assumed to
consist of identical individuals, except for Law and
Dieckmann [\ref{orga5cb81b}]. In their work, the authors considered
interactions between populations of different species, within which the
individuals are homogeneous.

We are interested here in the extension of the competition phenomenon to
continuously heterogeneous populations, i.e., populations where the
characteristics of the individuals can be continuously distributed. Different
models of competition were confronted with experimental data by Schneider
et al. [\ref{org07ae0a9}]. The authors used a hierarchical Bayesian framework to
conduct their analyses, with inter-individual variability of growth parameters
within the population. In this paper, we wish to link a dynamical system of the
form studied by Schneider et al. and the BPDL process, in particular, the
analysis of the system's behavior when the size of the population tends towards
infinity.

Our contributions are the following ones: first, we present a growth and competition
model proposed by Schneider et al., and we study the conditions of existence of
a global solution. In a second step, we establish the
mean-field dynamics associated with the system, corresponding to the limit when
the population size tends to infinity. Finally, we give a methodology to
simulate the evolution over time of the mean-field distribution.

\section{Growth and competition model}
\label{sec:org8e07c06}

Competition models for light can provide a more or less accurate description of
plant morphology, according to the objectives of modellers [\ref{org0f9a722}].  The
estimation of the influence of plant organs and compartments over the local
light environment is necessary to account for the variability in the
architecture of the aerial parts [\ref{org3e15225},\ref{org179f28c},\ref{orgfe59c74}]. In the
model considered by Schneider et al., initially designed to study competition in
a monospecific population of annual plants, arabidopsis (\emph{Arabidopsis
thaliana}), the plant morphology is only described by a characteristic
dimension, and this considerably alleviates the experimental protocol to monitor
the evolution of the population. The inter-plants competition is expressed via
an empirical potential, which depends only on the observed individual features,
referred to as a competition index [\ref{orgdece5cb}].

\subsection{A Gompertz growth function}
\label{sec:org713c0e7}

We recall in this paragraph the Gompertz growth function for a single plant with growth rate \(\gamma > 0\) and
asymptotic size \(S > s_m\), where \(s_m>0\) is a minimal size.

Let us start by introducing the dynamics of the plant in the absence of competition. The plant state is uniquely described by the variable \(s\),
representing one of its characteristic dimension, like the diameter of the
rosette of \emph{A. thaliana} in [\ref{org07ae0a9},\ref{orgd4e2f62}], or the diameter at breast
height of a tree in [\ref{org9115032},\ref{org1ddeaac},\ref{orgacc2f3a}]. The size of the
individual plant follows a Gompertz growth [\ref{orgfc92cc1}] towards an asymptotic
size \(S\) at a rate \(\gamma\):
\begin{equation}
\label{eq:org98dfe15}
\frac{ds}{dt}(t) = \gamma s(t)\left(\log\left(\frac{S}{s_m}\right)-
\log\left(\frac{s(t)}{s_m}\right)\right).
\end{equation}
In the above equation, \(S,\gamma\) are intrinsic parameters of the individual
plant. The variation of these parameters from one plant to another can be due to
genetic variability or micro-variations of the environment. \(s_m\) is the minimal
size of the plants, principally used as a normalization constant.
The differential equation (\ref{eq:org98dfe15}) can be solved analytically for
any given initial condition \(s^0\)
\begin{equation}
s(t) = S\left(\frac{s^0}{s_m}\right)^{e^{-\gamma t}}
\end{equation}
The Gompertz growth is a specific case of the Richards' growth, used to model
the evolution of the plant size without competition in [\ref{org960cc74}].

\subsection{Growth subject to competition}
\label{sec:org60dc92f}

In this section, we consider a heterogeneous population of plants represented by
the list \((s_i(t),x_i,S_i,\gamma_i)_{1\leq i\leq N}\), with \(N> 1\). The variable
\(x_i\) represents the position of the plant \(i\) in the plane \(\RR^2\),
\(S_i,\gamma_i\) are the individual parameters of the plant. Each individual at
a given time \(t\) can therefore be represented by a point in the space
\(\mathcal Z = \RR_+^*\times \RR^2\times \RR_+^*\times \RR_+\). \((x,S,\gamma)\) are individual specific parameters that are constant during the growth of the plant, and are in the space
\(\Theta\subset \RR^2\times \RR_+^*\times \RR_+\) of individual parameters. In what
follows, we use the
notation \(\theta\) for \((x,S,\gamma)\), \(\theta'\) for \((x',S',\gamma')\), etc. In
the next sections, some specific constraints are set on \(\Theta\), so
that the dynamics of the population meet particular conditions.

The growth of the plants in the population \((s_i(t),x_i,S_i,\gamma_i)_{1\leq i\leq N}\) is expressed as
a system of \(N\) differential equations: for all $i\in
\llbracket 1;N\rrbracket$, we assume that
\begin{equation}
\begin{aligned}
&\frac{ds_i}{dt}(t) = \gamma_i s_i(t)\left(\log\left(\frac{S_i}{s_m}\right)\left(
1-C_i^N(t)\right)-\log\left(\frac{s_i(t)}{s_m}\right)\right)\ ,\\
&\text{where }C_i^N(t) = \frac{1}{N-1}\sum_{j\neq i} C(s_i(t),s_j(t),\vert x_i-x_j\vert)\ , \text{and }\\
&C(s_i(t),s_j(t),\vert x_i-x_j\vert) = \frac{\log\left(\frac{s_j(t)}{s_m}\right)}{2R_M
\left(1+\frac{\vert x_i-x_j\vert^2}{\sigma_x^2}\right)}\left(1+\tanh\left(\frac{1}{\sigma_r}\log\left(\frac{s_j(t)}{s_i(t)}\right)\right)\right).
\end{aligned}
\end{equation}
\(C(s,s',\vert x-x'\vert)\) is the competition potential between two plants of sizes \(s,s'\)
and located at the distance \(\vert x-x'\vert \) of each other. \(C_i^N(t)\) is the
competition index exterted on the plant \(i\), resulting from an average of the
competition potential over all the other plants in the population. In
this model, competition is, therefore, to be understood as a negative
perturbation of the development a plant would theoretically have in optimal
conditions. The competition potential is composed of three factors that can be
interpreted separately:
\begin{enumerate}
\item \(\log(s_j(t)/s_m)/R_M\): the larger plant \(j\), the stronger the competition it
exerts on plant \(i\); \(R_M\) is chosen so that this term stays between 0 and 1;
\item \(1/2\left(1+\tanh\left(\frac{1}{\sigma_r}\log(s_j(t)/s_i(t))\right)\right)\):
the larger plant \(i\) in comparison with plant \(j\), the weaker the competition
exerted on plant \(i\) by \(j\); the parameter \(\sigma_r\) monitors the effect of
the relative size;
\item \(\frac{1}{1+\frac{\vert x_i-x_j\vert ^2}{\sigma_x^2}}\): the further apart plants \(i\)
and \(j\) are, the less competition there is; the parameter \(\sigma_x\)
monitors the rate of the spatial decrease of the competition.
\end{enumerate}
\(C\) is by design supposed to be a proportion between 0 and 1. This property is
true if we can prove that the plant size does not exceed some maximal size \(S_M\)
and if \(R_M\geq \displaystyle \log\left(\frac{S_M}{s_m}\right)\). In the next
section, we establish a sufficient condition for this property to hold.

The competition potential considered in this article is one of the model studied
in Schneider et al. [\ref{org07ae0a9}], but the authors have considered a
non-normalized potential. The structure of the potential is very similar to the
one used in Adams et al. [\ref{org9115032}], with the difference that the spatial
decrease of competition is expressed with a Gaussian kernel. This difference
has no impact on the results discussed in the next sections. This model was
chosen for its smoothness with respect to plant state and distance between
competitors, and it was shown to capture well the dynamics of experimental
observations in Schneider et al. [\ref{org07ae0a9}]. As for the competition index
based on the intersection of the region of influence, that is used as an
illustration in Campillo and Johannides [\ref{org960cc74}], it does not have an
analytical expression, which makes the analysis of large population dynamics
more difficult. 

This competition model is rather simple, as it does not take into account
the architecture of the plant and its interaction with the resources of its
environment, but shows good robustness properties. We refer the reader to the models studied in Cournède et
al. [\ref{org179f28c}] and Beyer et al. [\ref{orgfe59c74}] for more complex and architecture-based representation of the competition for light in plants. 

\subsection{Existence and uniqueness of a global solution}
\label{sec:org812deaf}

We study the properties of the solutions of the differential system induced by
the competition. In particular, we establish sufficient conditions on the
parameters and on the initial conditions to ensure that the system has a
dynamics consistent with the biology: no finite-time blow-up (the solution must
be globally defined), the sizes of the plant must remain positive, and all
competition indices must remain between 0 and 1. To this purpose, we will use
the following lemma which is related with Grönwall lemma.

\begin{lemma}
Let \(y:[0;T)\rightarrow \RR\) be a continuously differentiable function
defined over an interval \([0;T)\) where \(T\in \bar \RR_+\). If there exists $a\in
\RR$ and \(b\neq 0\) such that for all \(t\in [0;T)\) $\displaystyle \frac{dy}{dt}\leq 
a-by(t)$ (resp. \(\displaystyle \frac{dy}{dt}\geq a-by(t)\)), then for all $t\in
[0;T)$, we have
\begin{equation}
y(t)\leq \text{(resp. }\geq ) \frac{a}{b} + \left(y(0)-\frac{a}{b}\right)e^{-bt}
\end{equation}  
\label{orgd765c7c}
\end{lemma}

\begin{proposition}
Let the initial configuration of the population
\((s_i^0,x_i,S_i,\gamma_i)_{1\leq i \leq N}\). If for all plant $i\in \llbracket 
1;N\rrbracket$, \(s_m < s_i^0 < S_i\), \(s_m< S_i< s_m e^{R_M}\) and \(\gamma_i >0\), 
then the dynamical system of \(N\) equations, defined by
\begin{equation}
\label{eq:org70d362f}
    \forall i \in \{1,...,N\},\left\{\begin{array}{l}
    s_i(0) = s_i^0\\
    \displaystyle \forall t\geq 0,\frac{ds_i}{dt}(t) = \gamma_i s_i(t)\left(\log
\left(\frac{S_i}{s_m}\right)\left(1-C_i^N(t)\right)-\log\left(\frac{s_i(t)}{s_m}
\right)\right)
    \end{array}\right.
\end{equation}
has a unique global solution, defined over \(\RR_+\). Moreover, for all \(t\geq 0\)
and for all \(i\in \llbracket 1;N\rrbracket\), \(0\leq C_i^N(t)\leq 1\).
\label{orgcab5114}
\end{proposition}

\begin{proof}:
Let \(t_m >0\) be the upper-bound of the defintion interval of the maximal
solution of equation (\ref{eq:org70d362f}). We consider the interval
\begin{equation}
\mathcal I = \left\{t\in [0;t_m)\vert \forall \tau \in [0;t],\forall i \in \llbracket 
1;N\rrbracket,~s_m<s_i(\tau)<S_i \right\}
\end{equation}
Let \(t^* = \sup\mathcal I\). For all \(t\in [0;t^*)\), by application of lemma
\ref{orgd765c7c}, we have the following
inequalities
\begin{equation}
\label{eq:orge560810}
s_m\left(\frac{s_i^0}{s_m}\right)^{e^{-\gamma_i t}}\leq s_i(t)\leq S_i\left(
\frac{s_i^0}{s_m}\right)^{e^{-\gamma_it}}.
\end{equation}
If \(t^* < +\infty\), then we can prove using inequality (\ref{eq:orge560810})
that the integral \(\displaystyle \int_0^{t^*} \frac{ds_i}{dt}(t)dt\) is
absolutely convergent for all \(i\in \llbracket 1;N\rrbracket\). Therefore, all
sizes \(s_i\) have a limit when \(t\rightarrow t^*\), and they can be extended by
continuity at time \(t^*\). At this time, we have for all \(i\), $s_m<s_i(t^*)<S_i$, which is in contradiction with the definition of \(t^*\). We conclude that
\(t^* = t_m = +\infty\). The inequalities on the competition indices \(C_i^N\) are
direct consequences of the fact that \(s_m < s_i(t)<S_i\) at all time \(t\geq 0\).
\end{proof}

We can rewrite the expression of the size \(s_i(t)\) using its competition index
to compare it with the case without competition.
\begin{equation}
s_i(t) = s_m\left(\frac{s^0_i}{s_m}\right)^{e^{-\gamma_i t}}\exp\left(\log\left(
\frac{S_i}{s_m}\right)\gamma_i\int_0^t(1-C^N_i(\tau))e^{\gamma_i(\tau-t)}d\tau\right)
\end{equation}
Therefore, the asymptotic behavior when \(t\rightarrow +\infty\) is mainly driven
by the term \(\displaystyle
\gamma_i\int_0^t(1-C^N_i(\tau))e^{\gamma_i(\tau-t)}d\tau\), which can be
understood as an average over time, with exponential weight \(\gamma_i e^{\gamma_i(\tau-t)}\), of the complement of the competition.  

\subsection{Simulation of the growth}
\label{sec:orgbb70bff}

\subsubsection{Example of an initial distribution}
\label{sec:orgee9fd66}
\label{org038c561}

In this section, we choose a parametric expression for the initial configuration
distribution \(\mu_0\). We are interested in a population having a spatial
distribution of individual parameters \(S\) and \(\gamma\), meaning that these
parameters are chosen with a high correlation with the position variable \(x\).

In keeping with Schneider et al. [\ref{org07ae0a9}] and Lv et al. [\ref{orgd4e2f62}], the
initial sizes of the plants are fixed to a constant \(\bar s = s_0>s_m\) over
the population, and the positions of the plants are distributed according to a
Gaussian distribution \(\mathcal N(0,L^2\mathrm I_2)\) for some distance
\(L\). Lv et al. [\ref{orgd4e2f62}] also consider a Poisson point-process
distribution of the plants over the plane. To simplify the subsequent analysis, we will consider that the
initial configurations of the individuals in the population are independent and
identically distributed, i.e., with a distribution of the form \(\mu_0^{\otimes N}\).

The distribution of the individual parameters \(S\) and \(\gamma\) is determined by
two parametric surfaces \(x\in \RR^2\mapsto\bar S(x)\in \RR_+\) and \(x\in \RR^2\mapsto \bar \gamma(x)\in \RR_+\) defined by
\begin{equation}
\begin{aligned}[t]
\bar S(x) = & S_0+(S_M-S_0)\exp\left(-\frac{1}{2}(x-x_1^S)^\tr H_1^S(x-x_1^S)\right)
\\
 & -(S_0-S_m)\exp\left(-\frac{1}{2}(x-x_2^S)^\tr H_2^S(x-x_2^S)\right)\\
\bar \gamma(x) = & \gamma_0+(\gamma_M-\gamma_0)\exp\left(-\frac{1}{2}(x-x_1^\gamma)^\tr H_1^\gamma(x-x_1^\gamma)\right)\\
& -(\gamma_0-\gamma_m)\exp\left(-\frac{1}{2}(x-x_2^\gamma)^\tr H_2^\gamma(x-x_2^\gamma)\right)
\end{aligned}
\end{equation}
In the above equations, the surfaces are parameterized by their offsets \(S_0\) or
\(\gamma_0\), the location of high values of the parameter by \(x_1^S\) or
\(x_1^\gamma\), the location of low values of the parameter by \(x_2^S\) or
\(x^\gamma_2\), typical high values by \(S_M\) or \(\gamma_M\), and typical low values by
\(S_m\) or \(\gamma_m = 0\). The four matrices \(H_1^S,H_2^S,H_1^\gamma,H_2^\gamma\) are
symmetric positive and they monitor the shape of the surface in the
neighborhoods of \(x_1^S\), \(x_2^S\), \(x_1^\gamma\), \(x_2^\gamma\) respectively.

Individual parameters \(S\) and \(\gamma\) are independent conditionally to
the position \(x\), with the following conditional distributions
\begin{equation}
S\vert x\sim \mathcal N_{[S_m;s_me^{R_M}]}(\bar S(x),\delta S^2), ~ \gamma\vert x\sim
\mathcal N_{[0;\gamma_M]}(\bar \gamma (x),\delta \gamma^2),
\end{equation}
where \(\mathcal N_{[a;b]}\) is a truncated Gaussian distribution over the segment
\([a;b]\).

In summary, the initial distribution chosen as an example for the simulation has
the following expression:
\begin{equation}
\begin{array}{ll}
\mu_0(ds,dx,dS,d\gamma) = &\delta_{s_0}(ds)\mathcal N_{[S_m;s_me^{R_M}]}(
\bar S(x),\delta S^2)(dS)\\ 
&\ \ \ \ \ \ \ \mathcal N_{[0;\gamma_M]}(\bar \gamma (x),\delta
\gamma^2)(d\gamma)
\mathcal N(0,L^2\mathrm I_2)(dx)
\end{array}
\end{equation}
\subsubsection{Simulation of a population subject to competition}
\label{sec:org58a8798}

\label{org46e8607}

We simulate the growth of a population of size \(N = 50\) individuals starting from
independent and identically distributed samples of the distribution \(\mu_0\)
defined in the previous section. We used a Runge-Kutta method of 5\(^{th}\)
order with 4\(^{th}\) order free interpolation [\ref{org5f8ecba}] implemented in
the package \texttt{DifferentialEquations.jl} of \texttt{Julia} language
[\ref{org32a4d27}] to solve the nonlinear system (\ref{eq:org70d362f}). The values
of the parameters chosen for the initial distribution and the competition
potential are given in Table \ref{org4cb584e}.

\begin{figure}[htbp]
\centering
\includegraphics[width=.9\linewidth]{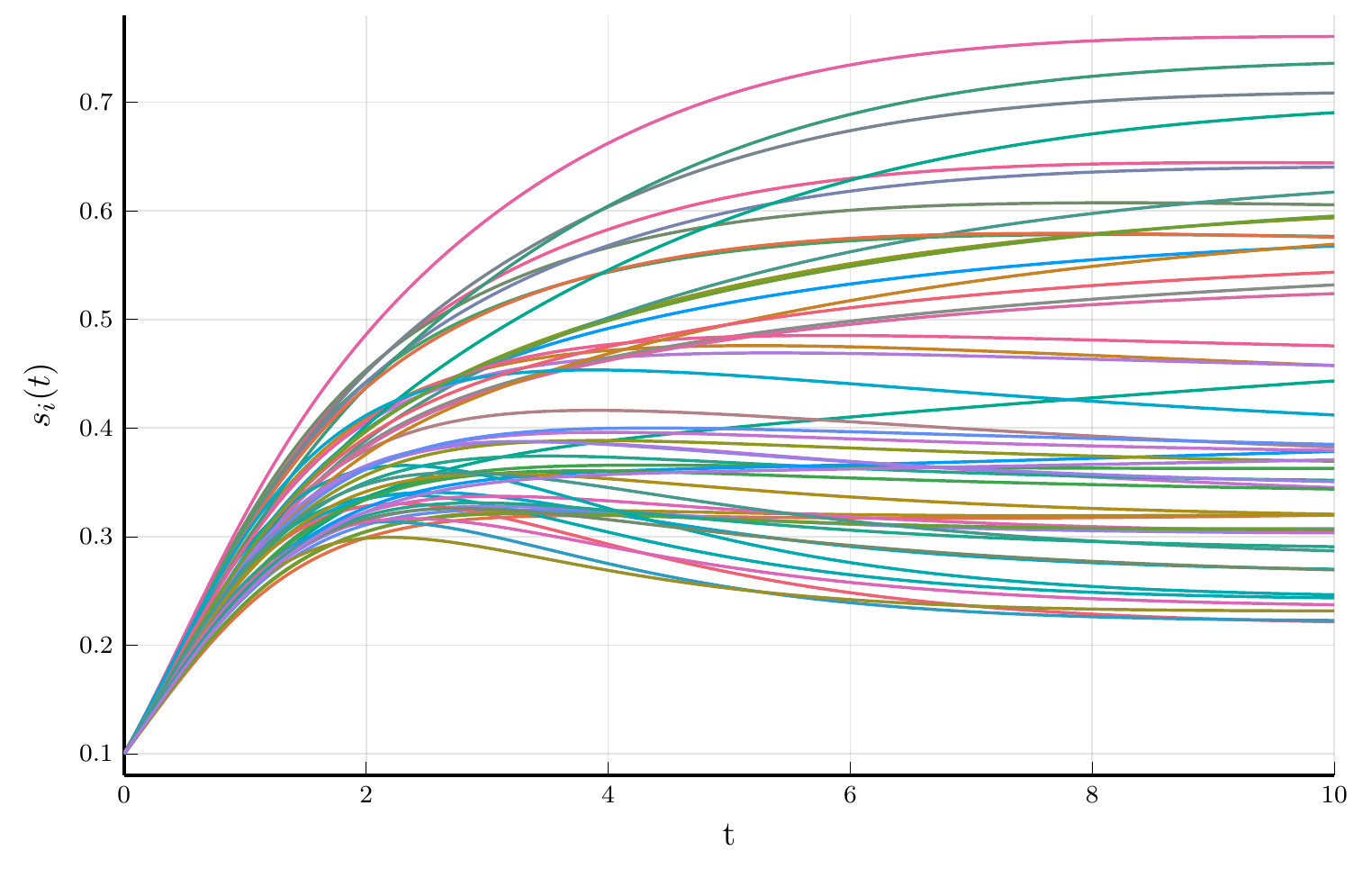}
\caption{\label{fig:org68c03a8}Set of individual trajectories \(t\in [0;10]\mapsto s_{1:N}(t)\) obtained with Julia package \texttt{DifferentialEquations.jl} for a population of size \(N=50\).}
\end{figure}

Figure \ref{fig:org68c03a8}  represents the evolution of a population of size
\(N = 50\) individuals over the time interval \([0;10]\) for a given initial
configuration. We can notice by looking at the final slopes of the growth curves
that the sizes of the different plants have, for a vast majority, reached
stationary sizes at time \(T = 10\), leading to think that the whole population
may have a stationary distribution.

\section{Mean-field model}
\label{sec:org053cfa3}

We study the asymptotic behavior of the system (\ref{eq:org70d362f}) when the size \(N\)
of the population tends towards \(+\infty\). More precisely, we would like to
characterize the growth of an individual plant in an infinitely crowded
population. This evolution is represented by the notion of characteristic flow,
which drives the individual and the global dynamics at the same
time. Qualitatively, the characteristic flow can be seen as the growth of a plant in interaction with a population represented by a probability distribution, instead of a list
of individuals like in the case of the differential system (\ref{eq:org70d362f}). We first introduce the notion of population empirical measure and empirical
flow.

\subsection{Population empirical measure and empirical flow}
\label{sec:org27dddd6}

\begin{definition}
Let \(z^0_{1:N} = (s_i^0,x_i,S_i,\gamma_i)_{1\leq i\leq N} \in \mathcal Z^N\) be
an initial configuration satisfying the same assumptions as in proposition
\ref{orgcab5114}. Let \(t\in \RR_+\mapsto (s_i(t))_{1\leq i\leq N}\in (\RR_+^*)^N\) be the solution of the system (\ref{eq:org70d362f}). We define as the
population empirical measure the trajectory taking values in the space \(\mathcal{P}(\mathcal Z)\) of probability measures over \(\mathcal Z\):
\begin{equation}
t\in \RR_+\mapsto \hat \mu_N(t,z_{1:N}^0)\in \mathcal P(\mathcal Z)
\end{equation}
where for all \(t\geq 0\),
\begin{equation}
\label{eq:org6c37da2}
\hat \mu_N(t,z_{1:N}^0)(dz) = \frac{1}{N}\sum_{i = 1}^N \delta(s_i(t),x_i,S_i,\gamma_i)(dz)
\end{equation}
where \(\delta(s_i(t),x_i,S_i,\gamma_i)\) is the Dirac distribution centered at
the point \((s_i(t),x_i,S_i,\gamma_i)\). 
\end{definition}
The population empirical measure can be understood as a uniform distribution
over the individuals in the population.
The dynamics of the population empirical measure can be characterized by a
function related to the semi-group of the differential system (\ref{eq:org70d362f}),
that we refer to as the empirical flow, denoted by \(\hat s_N\). The property
characterizing the empirical flow can be written as follows: for all \(t\geq 0\),
\begin{equation}
\hat \mu_N(t,z^0_{1:N})(ds,d\theta) = (\hat s_N(t,s,\theta),\theta)\#\hat \mu(0,z^0_{1:N})(ds,d\theta)
\end{equation}
 For
any function \(f:\mathcal X\rightarrow \mathcal Y\) and any probability measure
\(\mu\in \mathcal P(\mathcal X)\), the operator \(f(x)\#\mu(dx)\in \mathcal P(\mathcal Y)\)
is the pushforward measure of \(\mu\) by the function \(f\), defined for all bounded
and measurable function \(\varphi: \mathcal Y\rightarrow \RR\) by
\begin{equation}
\int_{\mathcal Y}\varphi(y)[f(x)\#\mu(dx)](dy) = \int_{\mathcal X}\varphi(f(x))\mu(dx)
\end{equation}

The next proposition gives the expression of the empirical flow \(\hat s_N\).
\begin{proposition}
Let \((s_i^0,\theta_i)_{1\leq i\leq N}\in \mathcal Z^N\) be an initial configuration of
the population satisfying the same assumptions as in proposition
\ref{orgcab5114}. Then for all \((s^0,x,S,\gamma)\in \mathcal Z\) such that
\(s_m<s^0<S\) and \(s_m< S< s_me^{R_M}\), the differential equation
\begin{equation}
\label{eq:org38ac69c}
\begin{aligned}
&\left\{\begin{array}{l}
s(0) = s^0\\
\displaystyle \forall t\geq 0, \frac{ds}{dt}(t) = \gamma s(t)\left[\log\left(\frac{S}{s_m}\right)
\left(1-\hat C_N(s(t),x,\hat \mu_N(t,z^0_{1:N}))\right)-\log\left(
\frac{s(t)}{s_m}\right)\right]
\end{array}\right.\\
\end{aligned}
\end{equation} 
where
\begin{equation*}
\hat C_N(s(t),x,\hat \mu_N(t,z^0_{1:N}))\!= \!\frac{N}{N-1}\!\int_{\RR_+^*\times \RR^2}\!\!\!\!\!\!\!\!\!C(s(t),s',\vert x-x'\vert )\hat \mu_N^{s,x}(t,z^0_{1:N})(ds',dx')-\frac{C(s(t),s(t),0)}{N-1}
\end{equation*}
has a unique solution, which is defined over \(\RR_+\).  In equation (\ref{eq:org38ac69c}),
we use the notation \(\hat \mu_N^{s,x}(t,z^0_{1:N})\) for the marginal
distribution of \(\hat \mu_N(t,z^0_{1:N})\) with respect to the variables $s,x$. We call 
empirical flow the
function associating the initial condition to the solution of the differential
equation and we denote:
\begin{equation}
(t,s^0,\theta)\in \RR_+\times \mathcal Z\mapsto \hat s_N(t,s^0,\theta,z^0_{1:N})
\in \RR_+^*.
\end{equation}
\label{org23105ac}
\end{proposition}

\begin{proof}:
We only need to prove that the maximal solution of the equation is bounded over
its definition interval, which will imply that the maximal solution is
global. We have, for all \(t\) in the definition interval, the following
inequalities on the competition term.
\begin{equation}
-\frac{\log\left(\frac{s(t)}{s_m}\right)}{2R_M(N-1)} \leq \hat C_N(s(t),x,\hat
\mu_N(t))\leq \frac{N}{N-1}-\frac{\log\left(\frac{s(t)}{s_m}\right)}{2R_M(N-1)}
\end{equation}
By implying lemma \ref{orgd765c7c}, we can obtain bounds on the solution \(t\mapsto
s(t)\).
\begin{equation}
s_m\exp\left(-\frac{2R_M}{2N-3}\right)\leq s(t)\leq s_m\exp\left(\frac{6N-5}{2N-
3}R_M\right)
\end{equation}
\end{proof}

The empirical flow is such that \(\forall i\in \llbracket 1;N\rrbracket, ~\hat s_N(t,s_i^0,\theta_i,z^0_{1:N}) = s_i(t)\) and therefore we have \((\hat s_N(t,s^0,\theta,z^0_{1:N}),\theta)\#\hat \mu_N(0,z^0_{1:N})(ds^0,d\theta) = \hat\mu_N(t,z^0_{1:N})(ds_t,d\theta)\).

Let us assume that the initial population empirical measure \(\hat \mu_N(0,z^0_{1:N})\) has a limit when \(N\rightarrow \infty\) for some metric. This
is the case, for instance, if the initial configuration of the population is
composed by independent samples drawn from some distribution \(\mu_0\), which is
the illustration we used for the simulation in Section
\ref{org46e8607}. In this case, we have indeed that
\begin{equation}
\hat \mu_N(0,z^0_{1:N})\xrightarrow[N\rightarrow \infty]{\mathcal D}\mu_0
\end{equation}
in distribution $\mu_{0}$-almost surely.

The differential equation characterizing
the empirical flow can be rewritten as follows:
\begin{equation}
\begin{aligned}
&\frac{\partial \hat s_N}{\partial t}(t,s^0,\theta,z^0_{1:N}) = \gamma \hat s_N(t,s^0,\theta,z^0_{1:N})\left(\log\left(\frac{S}{s_m}\right)\right.\\
&\times\left(1-\frac{N}{N-1}\int_{\mathcal Z}C(\hat s_N(t,s^0,\theta,z^0_{1:N}),\hat s_N(t,s',\theta',z^0_{1:N}),|x-x'|)\hat \mu_N(0,z^0_{1:N})(ds',d\theta')\right.\\
&\left.\left.+\frac{C(\hat s_N(t,s^0,\theta,z^0_{1:N}),\hat s_N(t,s^0,\theta,z^0_{1:N}),0)}{N-1}\right)-\log\left(\frac{\hat s_N(t,s^0,\theta,z^0_{1:N})}{s_m}\right)\right)
\end{aligned}
\end{equation}
From this expression, we can postulate that if the empirical flow \(\hat s_N\) has
a limit \(s_\infty\) in some sense when the size of the population tends to
infinity, then this limit must satisfy the following equation:
\begin{equation}
\label{eq:org49ac8ae}
\begin{aligned}
&\frac{\partial s_\infty}{\partial t}(t,s^0,\theta) = \gamma s_\infty(t,s^0,\theta)\left(\log\left(\frac{S}{s_m}\right)\left(1-\int_{\mathcal Z}C(s_\infty(t,s^0,\theta),s_\infty(t,s',\theta'),\vert x-x'\vert )\mu_0(ds',d\theta')\right)\right.\\
&\left.-\log\left(\frac{s_\infty(t,s^0,\theta)}{s_m}\right)\right)
\end{aligned}
\end{equation}
The objectives of the next sections is to prove that the object \(s_\infty\)
exists and is uniquely defined, and that the convergence of \(\hat s_N\) towards
\(s_\infty\) holds for some specific metric.
\subsection{Existence and uniqueness of the mean-field flow}
\label{sec:org28b4582}

The equations of the type (\ref{eq:org49ac8ae}) characterizes the mean-field flow
\(s_\infty(t)\), in reference to the mean-field distribution \(\mu_\infty(t)\),
representing the propagation through time of the initial datum \(\mu_0\) by the
dynamics of an infinitely-crowded population. The literature studying this type
of integro-differential equations is vast. In our case, we were mainly inspired by the
work of Golse [\ref{org7f4b2af}] and Bolley et al. [\ref{org759b328}], who have studied
dynamics that share similar properities with system (\ref{eq:org70d362f}).

\begin{theorem}
Let \(\Theta = \RR^2\times [s_m;s_m e^{R_M}]\times [0;\gamma_M]\) where \(\gamma_M >0\).
Let \(\mu_0\in \mathcal P(\RR_+^*\times \Theta)\) such that
\begin{equation*}
\int_{\RR_+^*}\log\left(\frac{s}{s_m}\right)^2\mu_0^s(ds)<+\infty
\end{equation*}
Then there exists a unique function \(s_\infty: (t,s_0,\theta)\in \RR_+\times
\RR_+^*\times \Theta\mapsto s_\infty(t,s_0,\theta)\in \RR_+^*\) continuously
differentiable with respect to \(t\), such that for all \((s_0,\theta)\in
\RR_+^*\times \Theta\), we have
\begin{equation}
\label{eq:org5810f15}
\begin{aligned}
&\left\{\begin{array}{l}
s_\infty(0,s_0,\theta) = s_0\\
\displaystyle \frac{\partial s_\infty}{\partial t}(t,s_0,\theta) = \gamma s_{\infty}(t,s_0,\theta)\left(\log\left(\frac{S}{s_m}\right)\left(1-\mathcal C(s_\infty,t,s_0,\theta)\right)-\log\left(\frac{s_\infty(t,s_0,\theta)}{s_m}\right)\right)
\end{array}\right.\\
&\text{where }\mathcal C(s_\infty,t,s_0,\theta) = \int_{\RR_+^*\times \Theta}C(s_\infty(t,s_0,\theta),s_\infty(t,s_0',\theta'),\vert x-x'\vert )\mu_0(ds_0',d\theta')
\end{aligned}
\end{equation}
\end{theorem}

\begin{proof}

For convenience, we study the equation satisfied by the logarithm of the
size instead, i.e.,
\(\displaystyle r_\infty(t) = \log\left(\frac{s_\infty(t)}{s_m}\right)\), and
\(\displaystyle\nu_0(dr,d\theta) = \left(\log\left(\frac{s}{s_m}\right),\theta\right)\#
\mu_0(ds,d\theta)\).

\begin{equation}
\begin{aligned}
&\left\{\begin{array}{l}
\displaystyle\frac{\partial r_{\infty}}{\partial t}(t,r_0,\theta) =
\gamma\left(\log\left(\frac{S}{s_m}\right)\left(1-\mathcal
C_r(r_\infty,t,r_0,\theta)\right)-r_\infty(t,r_0,\theta)\right)\\
r(t_0,r_0,\theta) = r_0
\end{array}\right.\\
&\text{where
}\mathcal C_r(r_\infty,t,r_0,\theta) = \int_{\RR\times
\Theta}C_r(r_\infty(t,r_0,\theta),r_\infty(t,r_0',\theta'),\vert \theta_x-\theta'_x\vert )
\nu_0(dr_0',d\theta')\\
&\text{and
}C_r(r,r',\vert x-x'\vert ) =
\frac{r'}{2R_M\left(1+\frac{\vert x-x'\vert ^2}{\sigma_x^2}\right)}\left(1+
\tanh\left(\frac{r'-r}{\sigma_r}\right)\right)
\end{aligned}
\end{equation}
The above equation is equivalent to the following integral form:
\begin{equation}
\label{eq:org1aec3cd}
r_\infty(t,r_0,\theta) = r_0 +\gamma\int_{t_0}^t\left(\log\left(\frac{S}{s_m}
\right)(1-\mathcal C_r(r_\infty,\tau,r_0,\theta))-r_\infty(\tau,r_0,\theta)
\right)d\tau
\end{equation}

The outline of the proof is as follows:
\begin{enumerate}
\item we show that equation (\ref{eq:org1aec3cd}) admits a solution in a specific space of
continuous functions defined over a small time interval \([t_0-\delta
   t;t_0+\delta t]\).
\item We derive upper-bounds of the solution with respect to spatial and time
variables.
\item We show that the solution of (\ref{eq:org1aec3cd}) is uniquely defined over its
definition interval
\item A semi-group structure is proved for the solutions of (\ref{eq:org1aec3cd}).
\item We prove that \(\RR_+\) is included in the definition interval of the maximal
solution associated with the initial data \((0,\nu_0)\).
\end{enumerate}

\underline{Local existence}

Let \(t_0\in \RR\), \(r^*>0\), \(\delta t > 0\). Let us define the following
function over the space of continuous functions over \([t_0-\delta t;t_0+\delta
t]\times \RR\times \Theta\):
\begin{equation*}
\begin{aligned}
\rho\in \mathcal C^0([t_0-\delta t;t_0+\delta t]&\times \RR\times \Theta
\rightarrow \RR)\\
&\longmapsto \Vert \rho\Vert _{t_0,\delta t} = \sup\left\{\frac{\vert \rho(t,r,
\theta)\vert }{r^*+\vert r\vert },(t,r,\theta)\in [t_0-\delta t;t_0+\delta t]\times \RR\times
\Theta\right\}\in \bar\RR
\end{aligned}
\end{equation*}

The space of continuous functions
\(\mathcal{SL}_{t_0,\delta t} = \left\{\rho\in \mathcal C^0([t_0-\delta t;t_0+\delta t]\times \RR\times \Theta\rightarrow \RR)\vert ~\Vert \rho\Vert _{t_0,\delta t}< \infty\right\}\)
is a Banach space for the norm
\(\rho \in \mathcal{SL}_{t_0,\delta t}\mapsto \Vert \rho\Vert _{t_0,\delta t}\in \RR_+\). Over
\(\mathcal{SL}_{t_0,\delta t}\), we define the functional:
\begin{equation}
\begin{aligned}
&(\rho,t,r_0,\theta)\in \mathcal{SL}_{t_0,\delta t}\times [t_0-\delta t;t_0+
\delta t]\times \RR
\times \Theta\\
&\longmapsto \Phi(\rho,t,r_0,\theta) = r_0 + \gamma\int_{t_0}^t
\left(\log\left(\frac{S}{s_m}\right)(1-\mathcal C_r(\rho,\tau,r_0,\theta))-
\rho(\tau,r_0,\theta)\right)d\tau
\end{aligned}
\end{equation}
We have the following inequalities on the function \(\Phi\), for
\(\rho_1,\rho_2\in \mathcal{SL}_{t_0,\delta t}\)
\begin{equation*}
\begin{aligned}
&\Vert \Phi(\rho_1)\Vert _{t_0,\delta t}\leq \max\left\{1+\gamma_M\delta t
\Vert \rho_1\Vert _{t_0,\delta t}, \quad\gamma_M\delta t\frac{R_M + \Vert \rho\Vert _{t_0,\delta t}(
R_1(\nu_0)+r^*)}{r^*}\right\}\\
&\Vert \Phi(\rho_1)-\Phi(\rho_2)\Vert _{t_0,\delta t}\leq \gamma_M\delta t\Vert \rho_1-
\rho_2\Vert _{t_0,\delta t}\max\left\{1+\frac{\Vert \rho_1\Vert _{t_0,\delta t}}{2R_M
\sigma_r},\right.\\
&\left.\frac{(2R_M\sigma_r+\Vert \rho_1\Vert _{t_0,\delta t}R_1(\nu_0))r^* + 4R_1(\nu_0)
R_M\sigma_r+R_MR_2(\nu_0)(\Vert \rho_1\Vert _{t_0,\delta t}+\Vert \rho_2\Vert _{t_0,\delta t})}{
2R_M\sigma_r r^*}\right\}\\
&\text{with for }d \in \{1,2\},\quad R_d(\nu_0) = \int_{\RR}(r^*+\vert r\vert )^d\nu_0^r(dr)
\end{aligned}
\end{equation*}
Therefore, there exists \(\delta t> 0\) and \(\rho_m > 1\) such that

\begin{itemize}
\item
  for all \(\rho\in \mathcal{SL}_{t_0,\delta t}\) satisfying \(\Vert \rho\Vert _{t_0,\delta t}\leq \rho_m\) we
  have \(\Vert \Phi(\rho)\Vert _{t_0,\delta t}\leq \rho_m\);
\item
  \(\Phi\) is a contraction over the ball $\mathcal B_{t_0,\delta t,\rho_m} = \{\rho \in
\mathcal{SL}_{t_0,\delta t}\vert ~\Vert \rho\Vert _{t_0,\delta t}\leq \rho_m\}$.
\end{itemize}
By Banach fixed-point theorem, the map \(\Phi\) has a unique fixed-point
\(r_\infty\in \mathcal B_{t_0,\delta t,\rho_m}\).  Therefore, for all \(t_0\in \RR\)
and \(\nu_0\in \mathcal P_2(\RR\times \Theta)\), there exists an interval \(\mathcal I\) containing \(t_0\) and a function
\(r_\infty(t_0,\nu_0)\) satisfying the equation (\ref{eq:org1aec3cd}) for all
\((t,r_0,\theta)\in \mathcal I\times \RR\times \Theta\).

\underline{Propagation of the moment of the initial distribution}

Let \((r_{\mathcal I},\mathcal I)\) a solution of the equation (\ref{eq:org1aec3cd}) over
an interval \(\mathcal I\) containing \(t_0\). Then let us prove that for all \(t\in
\mathcal I\),
\begin{equation}
\Vert r_{\mathcal I}(t)\Vert  = \sup\left\{\frac{\vert r_{\mathcal I}(t,r_0,\theta)\vert }{r^*+\vert r_0\vert },(r_0,\theta)\in \RR
\times \Theta\right\}<+\infty
\end{equation}
We have the following inequality on the flow \(r_{\mathcal I}\) for all \(t\in
\mathcal I\) such that for all \(t\geq t_0\), $r\in \RR$, \(\theta \in\) \(\Theta\):
\begin{equation}
\vert r_{\mathcal I}(t,r_0,\theta)\vert \leq \vert r_0\vert +\gamma\int_{t_0}^t\left[\log\left(
\frac{S}{s_m}\right)\left(1+\int_{\RR\times \Theta}\frac{\vert r_{\mathcal I}(\tau,
r_0',\theta')\vert }{R_M}\nu_0(dr_0',d\theta')\right) + \vert r_{\mathcal I}(\tau,r_0,
\theta)\vert \right]d\tau
\end{equation}
We therefore need to upper-bound the first-order moment.
\begin{equation}
\begin{aligned}
&\int_{\RR\times \Theta}\vert r_{\mathcal I}(t,r_0,\theta)\vert \nu_0(dr_0,d\theta)\leq
\int_{\RR}\vert r_0\vert \nu_0^r(dr_0) + (t-t_0)\int_{[s_m;s_me^{R_M}]\times [0;
\gamma_M]}\!\!\!\!\!\!\!\!\!\!\!\!\!\!\!\!\!\!\!\!\!\!\!\!\!\!\!\!\!\!\!\!\!\!\gamma\log\left(S/s_m\right)\nu_0^{S,\gamma}(dS,d\gamma)\\
&+2\gamma_M\int_{t_0}^t\int_{\RR\times \Theta} \vert r_{\mathcal I}(\tau,r_0,\theta)
\vert \nu_0(dr_0,d\theta)d\tau
\end{aligned}
\end{equation}
By Grönwall lemma, there exists \(\bar R>0\), a constant not depending on $r_0$ and $t$, such that
\begin{equation}
\int_{\RR\times \Theta}\vert r_{\mathcal I}(t,r_0,\theta)\vert \nu_0(dr_0,d\theta)\leq
\left(\int_{\RR}\vert r_0\vert \nu_0^r(dr_0)+\bar R \right)e^{2\gamma_M(t-t_0)}-\bar R
\end{equation}
By applying Grönwall lemma a second time on \(\vert r_{\mathcal I}(t,r_0,\theta)\vert \), we
obtain that there exists constants \(\bar r_1,\bar r_2,\bar r_3 > 0\) such that, for all \(t\in
\mathcal I\cap[t_0;+\infty), r_0\in \RR,\theta\in \Theta\),
\begin{equation}
\label{eq:orgda9aba2}
\vert r_{\mathcal I}(t,r_0,\theta)\vert \leq (\vert r_0\vert +\bar r_1)e^{\gamma_M(t-t_0)}+
\bar r_2e^{2\gamma_M(t-t_0)}+\bar r_3
\end{equation}
So \(\Vert r_{\mathcal I}(t)\Vert <+\infty\). We apply the same reasoning to extend this
inequality for all \(t\in \mathcal I\cap (-\infty;t_0]\). In particular, we have,
for all \(t\in \mathcal I\),
\begin{equation}
\int_{\RR\times \Theta}\vert r_{\mathcal I}(t,r_0,\theta)\vert ^2\nu_0(dr_0,d\theta)<\infty
\end{equation}

\underline{Local uniqueness}

We consider the function on the space of continous functions
\begin{equation}
\rho\in \mathcal C^0(\RR\times \Theta\rightarrow\RR)\longmapsto \Vert \rho\Vert  =
\sup\left\{\frac{\vert \rho(r,\theta)\vert }{r^*+\vert r\vert },(r,\theta)\in \RR\times \Theta\right\}
\end{equation}
which is a norm over the Banach space \(\mathcal{SL} = \{\rho \in \mathcal C^0(\RR
\times \Theta\rightarrow \RR)\vert ~\Vert \rho\Vert <+\infty\}\).
Let \((r_{\mathcal I},\mathcal I)\) and \((r_{\mathcal J},\mathcal J)\) be two
solutions of the equation (\ref{eq:org1aec3cd}) on the the intervals \(\mathcal I\) and
\(\mathcal J\) respectively, associated with the same initial data \((t_0,\nu_0)\),
with \(t_0\in \mathcal I\cap \mathcal J = \mathcal J\subset \mathcal I\). Then we
can prove that the restriction \(r_{\mathcal I\vert \mathcal J}\) of \(r_{\mathcal I}\)
is equal to \(r_{\mathcal J}\). Let us consider the interval
\(\mathcal J' = \{t\in
\mathcal J\vert  t\geq t_0,\forall \tau\in [t_0;t],~r_{\mathcal I}(\tau) =
r_{\mathcal J}(\tau)\}\) and let \(t^* = \sup\mathcal J'\). Then \(t^* \geq \delta t>
0\)  by
uniqueness of the fixed point of the map \(\Phi\) previously defined. If \(t^*<\sup\mathcal J\), then
\(t^*\) is finite, and we obtain by continuity of the norm \(\Vert \cdot\Vert \) that
\(r_{\mathcal I}(t^*) = r_{\mathcal J}(t^*)\). Then for any \(\delta t^*> 0\) such
that \(t^*+\delta t^*\in \mathcal J\), there exists \(K(t^*,\delta t^*)>0\)
verifying for all \(t\in [t^*;t^*+\delta t^*]\) 
\begin{equation*}
\Vert r_{\mathcal I}(t) - r_{\mathcal J}(t)\Vert \leq K(t^*,\delta t^*)\int_{t^*}^t\Vert 
r_{\mathcal I}(\tau)-r_{\mathcal J}(\tau)\Vert d\tau
\end{equation*}
It follows that \(r_{\mathcal I}(t) = r_{\mathcal J}(t)\) for all \(t\in
[t^*;t^*+\delta t^*]\), which is in contradiction
with the definition of
\(t^*\). As a consequence, \(t^* = \sup \mathcal J\). We use the same reasoning to
prove that
\begin{equation}
\inf\{t\in \mathcal J\vert  \forall \tau\in [t;t_0],~r_{\mathcal I}(\tau) =
r_{\mathcal J}(\tau)\} = \inf\mathcal J
\end{equation}

The local existence and uniqueness implies the existence and uniqueness of the
maximal solution for equation (\ref{eq:org1aec3cd}) for any initial data
\((t_0,\nu_0)\). This means that there exists a solution and an interval
\((r_{\mathcal I},\mathcal I)\) such that for all \((r_{\mathcal J},\mathcal J)\)
associated with the same initial data, we have \(\mathcal J\subset \mathcal I\).

\underline{Semi-group structure of the flow}

Let \((r_{\mathcal I}(t_0,\nu_0),\mathcal I)\) be the maximal solution of the equation
(\ref{eq:org1aec3cd}) associated with the initial data \((t_0,\nu_0)\).
We use the following notation:
\begin{equation}
r_{\mathcal I}(t_0,\nu_0):(t,r,\theta)\in \mathcal I\times \RR\times \Theta
\mapsto r_{\mathcal I}(t,r,\theta;t_0,\nu_0)\in \RR
\end{equation}
Let \(t_1\in \mathcal I\) and let \(r_{\mathcal I}(t_1,\nu_{t_1})\) be the maximal
solution associated with \((t_1,\nu_{t_1})\) where \(\nu_{t_1}(dr_1,d\theta) = (r_{\mathcal I}(t_1,r_0,\theta;t_0,\nu_0),\theta)\#\nu_0(dr_0,d\theta)\)
is defined as the pushforward measure
from \(t_0\) to \(t_1\) by the flow.
Then, by uniqueness of the maximal solution, it follows that for all \(t\in
\mathcal I\), \(r_0\in \RR\), \(\theta\in \Theta\)
\begin{equation}
r_{\mathcal I}(t,r_0,\theta;t_0,\nu_0) = r_{\mathcal I}(t,r_{\mathcal I}(t_1,
r_0,\theta;t_0,\nu_0),\theta;t_1,\nu_{t_1})
\end{equation}

\underline{Globality of the maximal solution}

Let \((r_{\mathcal I}(0,\nu_0),\mathcal I)\) be the maximal solution associated
with \((0,\nu_0)\). Then let us prove that \(\RR_+\subset \mathcal I\). Let \(t^* =
\sup \mathcal I\). If \(t^*<+\infty\), then
\(t^*\) cannot be in \(\mathcal I\) as it would contradict the maximality of the
solution \((r_{\mathcal I},\mathcal I)\). But even if \(t^*\notin \mathcal I\), we
still obtain a contradiction with the maximality of the solution by considering
the solution associated with the initial data \((t^*,\nu_{t^*})\) where
\begin{equation}
\begin{aligned}
&\nu_{t^*} = (\tilde r_{\mathcal I}(t^*,\cdot,\cdot),\mathrm{Id}_\Theta)\#\nu_0\\
&\text{with }\tilde r_{\mathcal I}(t^*,r_0,\theta) = r_0 +\int_0^{t^*}
\frac{\partial r_{\mathcal I}}{\partial t}(t,r_0,\theta;0,\nu_0)dt
\end{aligned}
\end{equation}
The above integral is absolutely convergent thanks to the upper-bound
(\ref{eq:orgda9aba2}).

\underline{Conclusion}

By making the reverse change of variable \(\displaystyle s_\infty =
s_me^{r_\infty}\), we conclude that there exists a unique function \(s_\infty\)
defined over \(\RR_+\) satisfying the equation (\ref{eq:org5810f15}).
\end{proof}

We define the mean-field distribution of the population as the pushforward
measure of the initial distribution \(\mu_0\) by the mean-field flow.
\begin{equation}
\forall t\geq 0,\quad \mu_\infty(t)(ds_t,d\theta) =
(s_\infty(t,s_0,\theta),\theta)\# \mu_0(ds_0,d\theta)
\end{equation}
Using equation (\ref{eq:org5810f15}), we can prove that the distribution \(\mu_\infty\) is
the unique solution in the sense of the distribution of the following partial
differential equation:
\begin{equation}
\left\{\begin{array}{ll}
         \displaystyle \frac{\partial\mu_\infty}{\partial t}(t,ds,d\theta) + \frac{\partial}{\partial s}\left[\gamma s\left(\log\left(\frac{S}{s_m}\right)\left(1- \int_{\RR_+^*\times \RR^2}\!\!\!\!\!\!\!\!\!\!\!\!C(s,s',\vert x-x'\vert \mu^{s,x}_{\infty}(t,ds',dx')\right)\right)\mu_\infty(t,ds,d\theta)\right] = 0\\
\mu_\infty(0,ds,d\theta) = \mu_0(ds,d\theta)
\end{array}\right.
\end{equation}
Conversely, the distribution of the random variable \(s_\infty(t,s_0,\theta)\),
with \((s_0,\theta)\) being distributed according to \(\mu_0\), is the marginal
\(\mu_\infty^s(t)\). In the next section, we prove the convergence of \(\hat \mu_N(t,z^0_{1:N})\) towards \(\mu_\infty(t)\), which proves in the same time the
convergence of \(\hat s_N(t,s_0,\theta,z^0_{1:N})\) towards
\(s_\infty(t,s_0,\theta)\) when \(N\rightarrow +\infty\).

\subsection{Convergence of the empirical flow towards the mean-field flow}
\label{sec:org970a447}
We establish in this section that the population empirical measure \(\hat \mu_N(t,z^0_{1:N})\) converges almost surely towards \(\mu_\infty(t)\) for the
Wasserstein metric \(\mathcal W_p\), which is defined  (Villani [\ref{orgb06d6e2}]),
for any distributions \(\mu_1,\mu_2\in \mathcal P(\mathcal Z)\), as
\begin{equation}
\mathcal W_p(\mu_1,\mu_2)^p = \inf\left\{\iint_{\mathcal Z^2}m_{\mathcal Z}(z_1,z_2)^p\pi(dz_1,dz_2),\quad \pi\in \Pi(\mu_1,\mu_2)\right\}
\end{equation}
where \(m_{\mathcal Z}\) is a metric defined over the space \(\mathcal Z\) and
\(\Pi(\mu_1,\mu_2)\) is the set of couplings between the distributions
\(\mu_1,\mu_2\), i.e.,
\begin{equation}
\Pi(\mu_1,\mu_2) = \left\{\pi(dz_1,dz_2)\in \mathcal P(\mathcal Z^2)\vert ~\pi^{z_1}(dz_1) = \mu_1(dz_1)\text{ and }\pi^{z_2}(dz_2) = \mu_2(dz_2)\right\}
\end{equation}
The convergence of \(\hat \mu_N0,z^0_{1:N})\rightarrow \mu_0\) is propagated to
any time using the argument refered to in the literature as Dobrushin's
stability [\ref{orga691432}].

\begin{theorem}
Let \(\mu_0\in \mathcal P(\RR_+*\times \Theta)\) such that \(\displaystyle
\int_{\RR^2}\vert x\vert ^2\mu_0^x(dx)< +\infty\) and such that there exists
\(s_0^{\min}, s_0^{\max},S_m\) satisfying
\begin{equation}
\begin{aligned}
&s_m<s_0^{\min}<s_0^{\max}<S_m<s_me^{R_M}\\
&\proba\left\{s_0\sim \mu_0^s\vert s_0^{\min}<s_0<s_0^{\max}\right\} = 1\\
&\proba\left\{S\sim \mu_0^S\vert S_m<S<s_me^{R_M}\right\} = 1
\end{aligned}
\end{equation}
We consider a sequence \((s_n^0,x_n,S_n,\gamma_n)_{n\in \N^*}\) of independant and
identically distributed sample of the initial distribution \(\mu_0\). We build
from this sequence the sequence of empirical measures \((\hat \mu_N)_{N\in \N^*}\)
defined for all \(N\geq 1\) by equation (\ref{eq:org6c37da2}). Then we have the
following almost sure convergence:
\begin{equation}
\proba\left\{(s_n^0,x_n,S_n,\gamma_n)_{n\in \N^*}\sim \mu_0^{\otimes \infty}\vert \forall t\geq 0,~\lim_{N\rightarrow \infty}\mathcal W_1\left(\hat \mu_N(t),\mu_\infty(t)\right) = 0\right\}=1
\end{equation}

\end{theorem}

\begin{proof}
Let \(z_{1:N}^0\in ([s_0^{\min};s_0^{\max}]\times \Theta)^N\) be an initial
configuration. Let \(\pi_s\in \Pi\left(\hat \mu^s_N(0,z^0_{1:N}),\mu_0^s \right)\), \(\pi_x\in \Pi\left(\hat \mu_N^x(0,z^0_{1:N}),\mu_0^x\right)\),
  \(\pi_S\in \Pi\left(\hat \mu_N^S(0,z^0_{1:N}),\mu_0^S\right)\) and
  \(\pi_\gamma\in \Pi \left(\hat \mu_N^\gamma(0,z^0_{1:N}),\mu_0^\gamma\right)\)
  be a set of couplings associated with the marginal distributions of the
  empirical distribution and the initial distribution. From these couplings, we
  build the initial coupling \(\pi_0 = \pi_s\otimes \pi_x\otimes \pi_S\otimes
  \pi_\gamma\), which is in \(\Pi\left(\hat \mu_N(0,z^0_{1:N}),\mu_0\right)\).

We consider the empirical flow \(\hat s_N\) associated with the empirical
population measure \(t\mapsto \hat \mu_N(t,z^0_{1:N})\). For any \(t\geq 0\),
we define the coupling
\begin{equation}
\begin{aligned}
\pi_t(d\hat s_t,d\theta_1,ds_t,d\theta_2) &= (\hat s_N(t,\hat s_0,\theta_1,z^0_{1:N}),\theta_1,s_\infty(t,s_0,\theta_2),\theta_2)\#\pi_0(d\hat s_0,d\theta_1,ds_0,d\theta_2)\\
&\in \Pi\left(\hat \mu_N(t,z^0_{1:N}),\mu_\infty(t)\right)
\end{aligned}
\end{equation}
The space \(\RR_+^*\times \Theta\) is endowed with the metric \(m_{\mathcal Z}\)
defined by
\begin{equation}
  \begin{aligned}
    &m_{\mathcal Z}(s_1,\theta_1,s_2,\theta_2) = \frac{\vert s_1-s_2\vert +\vert S_1-S_2\vert }{s_m}+
    \frac{\vert x_1-x_2\vert }{\ell}+\tau_r\vert \gamma_1-\gamma_2\vert \\
    &= \frac{\vert s_1-s_2\vert }{s_m}+m_\Theta(\theta_1,\theta_2)
  \end{aligned}
\end{equation}
where \(\ell,\tau_r>0\) are arbitrary constants. For this metric and for any
 time \(t\geq 0\), the Wasserstein distance between distributions \(\hat \mu_N(t,
  z^0_{1:N})\) and \(\mu(t)\) is expressed as follows
\begin{equation}
  \begin{aligned}
    \mathcal W_1(\hat \mu_N(t,z^0_{1:N}),\mu(t)) = \inf&\left\{
    \iint_{(\RR_+^*\times \Theta)^2}\!\!\!\!\!\!\!\!\!\!\!\!\!m_{\mathcal Z}(s_1,\theta_1,s_2,\theta_2)
    \pi(d s_1,d \theta_1,d s_2,d \theta_2),\right.\\
    &\left.\phantom{\int}\pi\in \Pi\left(\hat \mu_N(t,z^0_{1:N}),\mu(t)\right)
    \right\}\\      
  \end{aligned}
\end{equation}
The coupling \(\pi_t\) provides an upper-bound of the Wasserstein distance at
time \(t\).
\begin{equation}
\label{eq:orga369bea}
  \begin{aligned}
    &\mathcal W_1(\mu_N(t,z_{1:N}^0),\mu(t))\leq \iint_{(\mathcal \RR_+^*
      \times \Theta)^2}\left(\frac{\vert s_1-s_2\vert }{s_m}+m_\Theta(\theta_1,\theta_2)
    \right)\pi_t(d s_1,d\theta_1,d s_2,d \theta_2)\\
    &\leq \frac{1}{s_m}\iint_{(\RR_+^*\times \Theta)^2}\!\!\!\!\!\!\!\!\!\!\!\!\!\vert \hat s_N(t,s_1,
    \theta_1)-s_\infty(t,s_2,\theta_2)\vert \pi_0(d s_1,d\theta_1,d s_2,
    d \theta_2)\\
    &+\iint_{\Theta^2}m_\Theta(\theta_1,\theta_2)\pi_0^\theta(d \theta_1,
    d \theta_2)
  \end{aligned}
\end{equation}
Let us focus on the first term of the upper-bound.
\begin{equation}
  \begin{aligned}
    &D^{\pi_0}_N(t) = \iint_{(\RR_+^*\times \Theta)^2}\!\!\!\!\!\!\!\!\!\!\!\!\!\!\!\!\vert \hat s_N(t,s_1,\theta_1)
    -s_\infty(t,s_2,\theta_2)\vert \pi_0(d s_1,d \theta_1,d s_2,d
    \theta_2)
  \end{aligned}
\end{equation}
As detailed in the appendix, section \ref{org776979e}, we can prove that
\begin{equation}
\label{eq:org34d318e}
    \begin{aligned}
      &D^{\pi_0}_N(t)\leq \frac{s_0^{\max}e^{R_M}R_M}{N-1}+\frac{tA(\hat
        \mu_N(0,z^0_{1:N}))}{N-1}+e^{R_M}\iint_{(\RR_+^*)^2}\vert s_1-s_2\vert \pi_0^s(
      d s_1,d s_2) \\
      &+ tB(\hat \mu_N^x(0,z^0_{1:N}),\mu_0^x)\sqrt{\iint_{\RR^4}\vert x_1-x_2\vert ^2
        \pi_0^x(d x_1,d x_2)}\\
      &+\alpha_S\iint_{[S_m;s_me^{R_M}]^2}\vert S_1-S_2\vert \pi_0^S(d S_1,
      d S_2)\\
      &+\alpha_\gamma t\iint_{[0;\gamma_M]^2}\vert \gamma_1-\gamma_2\vert \pi_0^\gamma(
      d \gamma_1,d \gamma_2)+\beta_N\int_0^tD^{\pi_0}_N(\tau)d \tau
    \end{aligned}
  \end{equation}
where the functionals \(A(\mu)\), \(B(\mu_1,\mu_2)\) have the following
expressions:
\begin{equation}
  \begin{aligned}
    &A(\mu) = \frac{1}{2R_M}\int_{\RR_+^*\times [S_m;s_me^{R_M}]\times
      [0;\gamma_M]}\!\!\!\!\!\!\!\!\!\!\!\!\!\!\!\!\!\!\!\!\!\!\!\!\!\!\!\!\!\!\!\!\!\!\!\!\!\!\!\!\!\!\!\!\!\!\!\left(\gamma Ss/s_m\log(s/s_m)+s_0^{\max}e^{R_M}R_M
    \gamma\log(S/s_m)\right)\mu^{s,S,\gamma}(d s,d S,d \gamma)\\
    &B(\mu_1,\mu_2) = \frac{s_0^{\max}e^{R_M}R_M\gamma_M}{\sigma_x^2}\left(
    2\int_{\RR^2}\vert x\vert \mu_2(d x)+\sqrt{2\int_{\RR^2}\vert x\vert ^2\mu_2(d x)+
      2\int_{\RR^2}\vert x\vert ^2\mu_1(d x)}\right.\\
    &\left.+\int_{\RR^2}\sqrt{2\int_{\RR^2}\vert x\vert ^2\mu_1
      (d x)+2\int_{\RR^2}\vert x\vert ^2\mu_1(d x)-4x'.\left(\int_{\RR^2}x
      \mu_1(d x)+\int_{\RR^2}x\mu_2(d x)\right)}\mu_1(d x')\right)
  \end{aligned}
\end{equation}
and the coefficients \(\alpha_S,\alpha_\gamma,\beta_N\) are
\begin{equation}
\label{eq:org36cbf9f}
  \begin{aligned}
    &\alpha_S = (s^0_{\max}/S_m)\\
    &\alpha_\gamma = s^0_{\max}\log\left(\frac{s_0^{\max}}{s_m}\right)e^{R_M}+
    s^0_{\max}e^{R_M}R_M\\
    &\beta_N = \frac{s_0^{\max}e^{R_M}R_M\gamma_M}{s_m^N\sigma_r}(1+\sigma_r/
    R_M+1/2)
  \end{aligned}
\end{equation}
We can find in the appendix, section \ref{org776979e}, the detailed
derivation of the upper-bound
of \(D^{\pi_0}_N(t)\) in inequation (\ref{eq:org34d318e}). By Grönwall lemma, we
obtain that
\begin{equation}
\label{eq:org2901026}
  \begin{aligned}
    &D^{\pi_0}_N(t)\leq \frac{1}{\beta_N}\left(\frac{A(\hat \mu_N(0,z^0_{1:N}))}{N-1}+B(\hat \mu_N^x(0,z^0_{1:N}),\mu_0^x)\sqrt{\iint_{\RR^4}\vert x_1-x_2\vert ^2\pi_0^x(d x_1,d x_2)}\right.\\
    &\left.+\alpha_\gamma\iint_{[0;\gamma_M]^2}\vert \gamma_1-\gamma_2\vert \pi_0^\gamma(d \gamma_1,d \gamma_2)\right)\left(e^{\beta_N t}-1\right)+\left(\frac{s_0^{\max}e^{R_M}R_M}{N-1}\right.\\
    &+e^{R_M}\iint_{[s_0^{\min};s_0^{\max}]^2}\vert s_1-s_2\vert \pi_0^s(d s_1,d s_2)\left.+\alpha_S\iint_{[s_m;s_me^{R_M}]^2}\vert S_1-S_2\vert \pi_0^S(d S_1,d S_2)\right)e^{\beta_Nt}
  \end{aligned}
\end{equation}
By gathering the terms from inequalities (\ref{eq:org2901026}) and (\ref{eq:orga369bea}), we obtain the following upper-bound on the Wasserstein
distance between the empirical distribution and the mean-field distribution
at time \(t\)
\begin{equation}
  \begin{aligned}
    &\mathcal W_1(\hat \mu_N(t,z^0_{1:N}),\mu(t))\leq \frac{e^{R_M}}{s_m}\mathcal
    W_1(\hat \mu_N^s(0,z^0_{1:N}),\mu_0^s)e^{\beta_Nt}\\
    &+\left(\frac{B(\hat\mu_N^x(0,z^0_{1:N}),\mu_0^x)}{s_m\beta_N}\left(
    e^{\beta_Nt}-1\right)+\frac{1}{\ell}\right)\mathcal W_2(\hat \mu_N^x(0,
    z^0_{1:N}),\mu_0^x))\\
    &+\frac{\left(\alpha_Se^{\beta_Nt}+1\right)}{s_m}\mathcal W_1(\hat \mu_N^S(0,
    z^0_{1:N}),\mu_0^S)+\left(\frac{\alpha_\gamma}{s_m}\left(e^{\beta_N t}-1
    \right)+\tau_r\right)\mathcal W_1(\hat \mu_N^\gamma(0,z^0_{1:N}),
    \mu_0^\gamma)\\
    &+\frac{1}{s_m(N-1)}\left(\frac{A_N(\hat\mu_N(0,z^0_{1:N}),\mu_0)}{
      \beta_N}(e^{\beta_Nt}-1)+s_0^{\max}e^{R_M+\beta_Nt}R_M\right)
  \end{aligned}
\end{equation}
From the law of large numbers, we have the following almost sure convergences
\begin{equation}
  \begin{aligned}
    &\proba(\Omega_A) = \proba\left\{z^0_{1:N}\sim \mu_0^{\otimes \infty}\mid
    \lim_{N\rightarrow\infty}A(\hat\mu_N(0,z^0_{1:N})) = A(\mu_0)
    \right\}=1\\
    &\proba(\Omega_B) = \proba\left\{z^0_{1:N}\sim \mu_0^{\otimes \infty}\mid
    \lim_{N\rightarrow \infty}B(\hat\mu^x_N(0,z^0_{1:N}),\mu_0^x) = B(\mu_0^x,
    \mu_0^x)\right\}=1
  \end{aligned}
\end{equation}
From Varadarajan's theorem [\ref{orgbaf95a3}], we have also
\begin{equation}
  \begin{aligned}
    &\proba(\Omega_s) = \proba\left\{z^0_{1:N}\sim \mu_0^{\otimes \infty}\mid
    \lim_{N\rightarrow \infty}\mathcal W_1(\hat \mu_N^s(0,z^0_{1:N}),\mu_0^s) = 0
    \right\} = 1\\
    &\proba(\Omega_x)=\proba\left\{z^0_{1:N}\sim \mu_0^{\otimes \infty}\mid
    \lim_{N\rightarrow \infty}\mathcal W_2(\hat \mu_N^x(0,z^0_{1:N},\mu_0^x))=0
    \right\}=1\\
    &\proba(\Omega_S) = \proba\left\{z^0_{1:N}\sim \mu_0^{\otimes \infty}\mid
    \lim_{N\rightarrow \infty}\mathcal W_1(\hat \mu_N^S(0,z^0_{1:N}),\mu_0^S) = 0
    \right\}=1\\
    &\proba(\Omega_\gamma) = \proba\left\{z^0_{1:N}\sim \mu_0^{\otimes \infty}
    \mid\lim_{N\rightarrow\infty}\mathcal W_1(\hat \mu_N^\gamma(0,z^0_{1:N}),
    \mu_0^\gamma)=0\right\}=1
  \end{aligned}
\end{equation}
As the intersection of all these events is of probability one, and as
\((\beta_N)_{N\geq 1}\) is convergent, we obtain the result we want to prove.
\end{proof}

\subsection{Simulation of the mean-field model}
\label{sec:org03172c0}
\label{orgeb2d575}

In this section, we detail a methodology to obtain numerical approximations of
the mean-field limit distribution in the specific case of the competition in the
BPDL system. An approximation of the mean-field characteristic flow is built by
estimating the competition potential. The time-evolution of the competition
potential is approximated by a piecewise constant function, and the spatial
dependency of the competition potential is learned by a parametric model. The
conditions of consistency of the numerical scheme are only considered
qualitatively. Consistent numerical methods to solve non-local transport
equations in low dimensions can be found in Carrillo et al. [\ref{org56fc075}],
Lafitte et al. [\ref{org8529c61}], Lagoutière and Vauchelet [\ref{org46ed14b}], but
these methods cannot be used directly in our case, due to the high dimension of
the phase space (which is equal to 5 in our case).

The expression of the mean-field flow as a function of the
competition potential is, for any \(\theta = (x,S,\gamma)\in \Theta\):
\begin{equation}
  \begin{aligned}
    &s_\infty(t,s,\theta) = s_m\left(\frac{s}{s_m}\right)^{e^{-\gamma t}}\exp
    \left(\gamma \log\left(\frac{S}{s_m}\right)\int_0^t(1-
    C_\infty(\tau,s,\theta))e^{\gamma (\tau-t)}d \tau\right)\\
    &\text{where }C_\infty(t,s,\theta) = \expect\left\{
    C(s_\infty(t,s,\theta),s_\infty(t,s',\theta'),\vert x-x'\vert ),~(s',\theta')\sim\mu_0
    \right\}\\
&\text{i.e.,}= \int_{\RR_+^*\times \Theta}\!\!\!\!\!\!\!\! C(s_\infty(t,s,\theta),s_\infty(t,s',\theta'),\vert x-x'\vert)\mu_0(ds',d\theta')
  \end{aligned}
\end{equation}
Let \(\{t_0 = 0,t_1,...,t_M = T\}\) be a subdivision of the observation interval
\([0;T]\) with regular time-step \(\Delta t\). We consider a piecewise constant
approximation of the competition potential.
\begin{equation}
\label{eq:orgb4caa17}
  \begin{aligned}
    &\hat s_\infty(t,s,\theta) = s_m\left(\frac{s}{s_m}\right)^{e^{-\gamma t}}
    \left(\frac{S}{s_m}\right)^{1-e^{-\gamma t}-\hat C_\infty(t,s,\theta)}\\
    &\hat C_\infty(t,s,\theta) = \gamma\sum_{k=0}^{M-1}\int_0^tC_k(s,\theta)
    \mathbb I\{t_k\leq t<t_{k+1}\}e^{\gamma(\tau-t)}d \tau
  \end{aligned}
\end{equation}
The \(C_k(s,\theta)\) are approximations of the competition potential exerted on a
plant of initial size \(s\) and of parameter \(\theta\) at time \(t_k\).

\underline{Approximation of the initial competition potential}

Let us build these approximations step by step. We start by the initial
competition potential, which has the following expression.
\begin{equation}
  \begin{aligned}
    C_\infty(0,s,\theta) = \expect\left\{\frac{\log(s'/s_m)}{2R_M\left(1+
      \frac{\vert x-x'\vert ^2}{\sigma_x^2}\right)}\left(1+\tanh\left(\frac{1}{
      \sigma_r}\log(s'/s)\right)\right),~(s',x')\sim \mu_0^{s,x}\right\}
  \end{aligned}
\end{equation}
In particular, we can notice that initially this potential does not depend on
parameters \(S\) and \(\gamma\). This expectation is not analytical in general, and
a natural way to obtain a consistent approximation of it is by resorting to
Monte-Carlo integration.

\begin{equation}
\label{eq:org994413a}
  \tilde C_\infty(0,s,x;s'_{1:N},x'_{1:N}) = \frac{1}{N}\sum_{i=1}^N
  C(s,s_i',\vert x-x_i'\vert )
\end{equation}
where \(s'_{1:N},x'_{1:N}\) are sampled from \(\mu_0^{s,x}\). This approximation is
more or less equivalent to simulating the microscopic dynamics that we know to
tend towards the mean-field dynamics. By doing so, there is not really any
computational advantage in the use of the mean-field flow, as it appears only as
an individual trajectory within a large enough population. One way to get rid of
the dependency with respect to the sample is to construct a parametric
approximation of the map in equation (\ref{eq:org994413a}). We can consider the
parametric family consisting of polynomial functions of some bounded
transformations of the variables \(s,x\).

\begin{equation}
\label{eq:orgbf76e42}
  \begin{aligned}
    &\mathfrak C_0 = \left\{(s,x,y,\beta)\in [s_0^{\min};s_0^{\max}]\times \RR^2
    \times \RR^{n(3,d)}\mapsto \beta.f^3_d(s,x,y),~d\in \N \right\}\\
    &\text{where }\forall d\in \N,~f^3_d(s,x,y) = \frac{q^3_d\left(\log(s/s_m),
      \arctan\left(\frac{x-\mu_x}{L_x}\right),\arctan\left(\frac{y-\mu_y}{L_y}
      \right)\right)}{1+\frac{(x-\mu_x)^2+(y-\mu_y)^2}{\sigma_x^2}}
  \end{aligned}
\end{equation}
In the above equation \(q^3_d\) is the polynomial feature function of three
variables and of degree \(d\). For instance, the polynomial feature function of
degree 2 with two variables is
\begin{equation}
  q^2_2(x_1,x_2) = (1,x_1,x_1^2,x_2,x_1x_2,x_2^2)\in \RR^6
\end{equation}
More generally, we denote by \(q^k_d\) the polynomial feature function with \(k\)
variables and of degree \(d\).
\begin{equation}
  \begin{aligned}
    &q^k_d(x_{1:k}) = \left(x_1^{\alpha_1}...x_k^{\alpha_k}
    \right)_{\alpha_{1:k}\in \mathrm{A}^k_d}\\
    &\mathrm{A}^k_d = \left\{\alpha_{1:k}\in \N^k\mid \alpha_1+...+\alpha_k\leq
    d\right\}
  \end{aligned}
\end{equation}
The cardinality of \(\mathrm{A}^k_d\) is a classical result of combinatorics.
\begin{equation}
  \mathrm{Card}(\mathrm{A}^k_d) = \sum_{\ell = 0}^d\binom{k+\ell-1}{k-1} = n(k,d)
\end{equation}
In equation (\ref{eq:orgbf76e42}), we have represented the position variable \((x,y)\)
by the bijective transformation \(\displaystyle \left(\arctan\left(\frac{x-\mu_x}{L_x}\right),\arctan\left(\frac{y-\mu_y}{L_y}\right)\right)\) where \((\mu_x,\mu_y)\) can
be chosen as the mean position, and \(L_x,L_y\) are typical lengths, such as the
standard deviation of \(x\) and \(y\), or the competition parameter \(\sigma_x\). With
this parametrization of the variable, we can express the competition potential
as a continuous function defined over a compact domain, that can be uniformly approximated
by a polynomial function thanks to the Stone-Weierstrass theorem. We also
divide the polynomial function by a factor \(\displaystyle \frac{1}{1+\frac{(x-\mu_x)^2+(y-\mu_y)^2}{\sigma_x^2}}\) to ensure that the approximation
has roughly the same behaviour as the target function when \(\vert x\vert \rightarrow
\infty\).

\(\beta\) is the vector of coefficients of the polynomial function at the
numerator. We can choose \(\beta\) so that it minimizes the quadratic risk between the
target function \(\tilde C_\infty(0,.)\) and the class \(\mathfrak C_0\) for a fixed
degree \(d\).
\begin{equation}
  \begin{aligned}
    &\beta_0^* = \underset{\beta \in \RR^{n(3,d)}}{\mathrm{argmin}}~\expect\{(
    \tilde C_\infty(0,s,x;s_{1:N},x_{1:N})-\beta\cdot f^3_d(s,x))^2,~(s,x)
    \sim\mu_0^{s,x}\},\\
    &\text{which is equivalent to find }\beta_0^*\text{ such that:}\\
    &\expect\left\{f^3_d(s,x)f^3_d(s,x)^\tr,(s,x)\sim\mu_0^{s,x}\right\}\beta_0^*
    = \expect\{\tilde C_\infty(0,s,x)f^3_d(s,x),(s,x)\sim\mu_0^{s,x}\}
  \end{aligned}
\end{equation}
The above linear system is not necessarily invertible, according to the
distribution \(\mu_0\): for instance, if \(\mu_0\) is a Dirac distribution, the
system is of rank 1. However, the system always admits at least one solution.
In practice, we consider a solution of the linear system
\begin{equation}
  \left(\sum_{k=1}^Kf^3_d(s_k,x_k)f^3_d(s_k,x_k)^\tr\right)\beta_0^* =
  \sum_{k=1}^K\tilde C_\infty(0,s_k,x_k)f^3_d(s_k,x_k)
\end{equation}
where \(s_{1:K},x_{1:K}\) is a training set consisting in a sample of \(\mu_0^{s,x}\)
independent of \(s'_{1:N},x'_{1:N}\). The approximation of the initial competition
potential is then
\begin{equation}
\label{eq:orgd75e2b9}
  \begin{aligned}
    &\hat C_0(s,x) = p_{[0;1]}(\beta\cdot f^3_d(s,x))\\
    &\text{where }p_{[0;1]}(x) = \max(\min(1,x),0)
  \end{aligned}
\end{equation}
We have incorporated in this reconstruction our knowledge on the boundedness of
the potential by simply projecting the values of the linear combination into
\([0;1]\). We can assess the relevance of this approximation by computing the
coefficient of determination over a testing set \((s^t_{1:K},x^t_{1:K},
\tilde C^t_{1:K})\).
\begin{equation}
  \begin{aligned}
    &R^2 = 1-\frac{\sum_{k= 1}^K(\tilde C^t_k-\hat C_0(s^t_k,x^t_k))^2}{
      \sum_{k=1}^K(\tilde C^t_k-m_{\tilde C^t})^2}\\
    &m_{\tilde C^t} = \frac{1}{K}\sum_{k=1}^K\tilde C^t_k
  \end{aligned}
\end{equation}
The coefficient of determination is used especially to calibrate the degree \(d\)
of the polynomial approximation, that need to be precise but also light in terms
of computation, as the dimension of the coefficient space \(n(k,d)\) increases like
the factorial function with \(d\).

\underline{Approximation of the subsequent competition potentials}

If the reconstruction is accurate enough, and if the sample size is large
enough, equation (\ref{eq:orgd75e2b9}) enables to sample \(\hat s_\infty(t,s,\theta)\)
for any \(t\) in the subinterval \([0;\Delta t]\), that is close to the actual
mean-field flow \(s_\infty(t,s,\theta)\). We then use the same methodology to
build the next approximation of the competition potential, but we need to
integrate in the arguments of the approximation the parameters \(S,\gamma\), that
have an influence on the competition potential at time \(\Delta t\).
\begin{equation}
  \begin{aligned}
    &\hat s_\infty(\Delta t,s,\theta) = s_m\left(\frac{s}{s_m}\right)^{e^{-\gamma \Delta t}}
    \left(\frac{S}{s_m}\right)^{(1-e^{-\gamma \Delta t})(1-\hat C_0(s,x))}\\
    &\tilde C_\infty(\Delta t,s,\theta;s'_{1:N},\theta'_{1:N}) =
    \frac{1}{N}\sum_{i=1}^NC(\hat s_\infty(\Delta t, s,\theta),\hat s_\infty(\Delta t,s'_i,
    \theta'_i),\vert x-x'_i\vert )
  \end{aligned}
\end{equation}
The class of functions used to approximate the above competition potential has
the same structure as \(\mathfrak C_0\), but with additional arguments.
\begin{equation}
  \begin{aligned}
    &\mathfrak C_1 = \left\{(s,x,y,S,\gamma,\beta)\in [s_0^{\min};s_0^{\max}]\times
    \Theta\times \RR^{n(5,d)}\mapsto \beta\cdot f^5_d(s,x,y,S,\gamma),~d\in \N
    \right\}\\
    &f^5_d(s,x,y,S,\gamma) =
    \frac{q^5_d\left(\log(s/s_m),
      \arctan\left(\frac{x-\mu_x}{L_x}\right),\arctan\left(\frac{y-\mu_y}{L_y}
      \right),\log(S/s_m),e^{-\gamma\Delta t}\right)}
         {1+\frac{(x-\mu_x)^2+(y-\mu_y)^2}{\sigma_x^2}}
  \end{aligned}
\end{equation}
Other choices of transformations for the variables \(S\) and \(\gamma\) are possible.
These specific transformations \(\log(S/s_m)\) and \(e^{-\gamma \Delta t}\) are used
here because they better describe the relationship between the competition
potential and the parameters \(S,\gamma\).

The identification of the linear combination coefficient \(\beta\) is done exactly
as previously by minimization of the square loss between the empirical potential
\(\tilde C_\infty(\Delta t,.,.;s'_{1:N},\theta'_{1:N})\) and the class \(\mathfrak C_1\) over a training set. The degree of the polynomial approximation is
calibrated by computing the coefficient of determination \(R^2\) over a testing
set. We similarly learn by recurrence the functions \(\hat C_1(s,\theta),...,
\hat C_{M-1}(s,\theta)\). To simplify the procedure, we choose the same value of
degree \(d\) for all the \(f^5_d\) at each time step \(\Delta t,...,T-\Delta t\),
that is potentially different from the degree chosen for the function \(f^3_d\),
used to approximate the initial potential. The final expression of the
approximated mean-field characteristic flows is obtained by computing the integral
over the time defining the reconstructed potential \(\hat C_\infty(t,s,\theta)\) in
equation (\ref{eq:orgb4caa17}).
\begin{equation}
  \begin{aligned}
    &\hat C_\infty(t,s,\theta) = \sum_{k=0}^{M-1}C_k(s,\theta)\left[\mathbb
      I\{t_k\leq t<t_{k+1}\}\left(1-e^{\gamma(t_k-t)}\right)\right.\\
      &\left.+\mathbb I\{t_{k+1}
      \leq t\}\left(e^{\gamma(t_{k+1}-t)}-e^{\gamma(t_k-t)}\right)\right]
  \end{aligned}
\end{equation}

When the sample size \(N\rightarrow \infty\), the
empirical potential converges to the mean-field potential uniformly and almost
surely. This empirical potential is well approximated by the functions of the
families \(\mathfrak C_0\) or \(\mathfrak C_1\) if the degree \(d\) is chosen large
enough.

\underline{Example of a simulation}

We applied this numerical scheme to the initial distribution \(\mu_0\) defined in
the section \ref{org038c561}. To be a in a more generic setting though, we consider
that the initial size is uniformly distributed over the interval
\([s_0^{\min};s_0^{\max}]\), instead of being constant equal to \(s_0\). The values
of the parameters chosen for this simulation are given in table
\ref{org018e7c3}, and the list of the reconstruction performance, computed in
terms of \(R^2\) are given in table \ref{orge12493f}.

The spatial variations of functions \(\bar S(x)\) and \(\bar \gamma(x)\) are
represented on figure \ref{fig:org3bba911}. The configuration of the
initial distribution is chosen in order to have four spatial regions
distinguishing the parameters: regions with high and low \(S\), and regions with
high and low \(\gamma\). In the absence of competition, we expect the mean size of
plants at a given position to converge to \(\bar S(x)\). We can notice on figure
\ref{fig:org0964585} that, due to the competition, the surface \(x\mapsto
\hat s_\infty(T,\bar s,x,\bar S(x),\bar \gamma(x))\) is quite different from
\(x\mapsto \bar S(x)\), and that the region where the plants remain small in
average is wider than in the case without competition.

\begin{figure}[htbp]
\centering
\includegraphics[width=.9\linewidth]{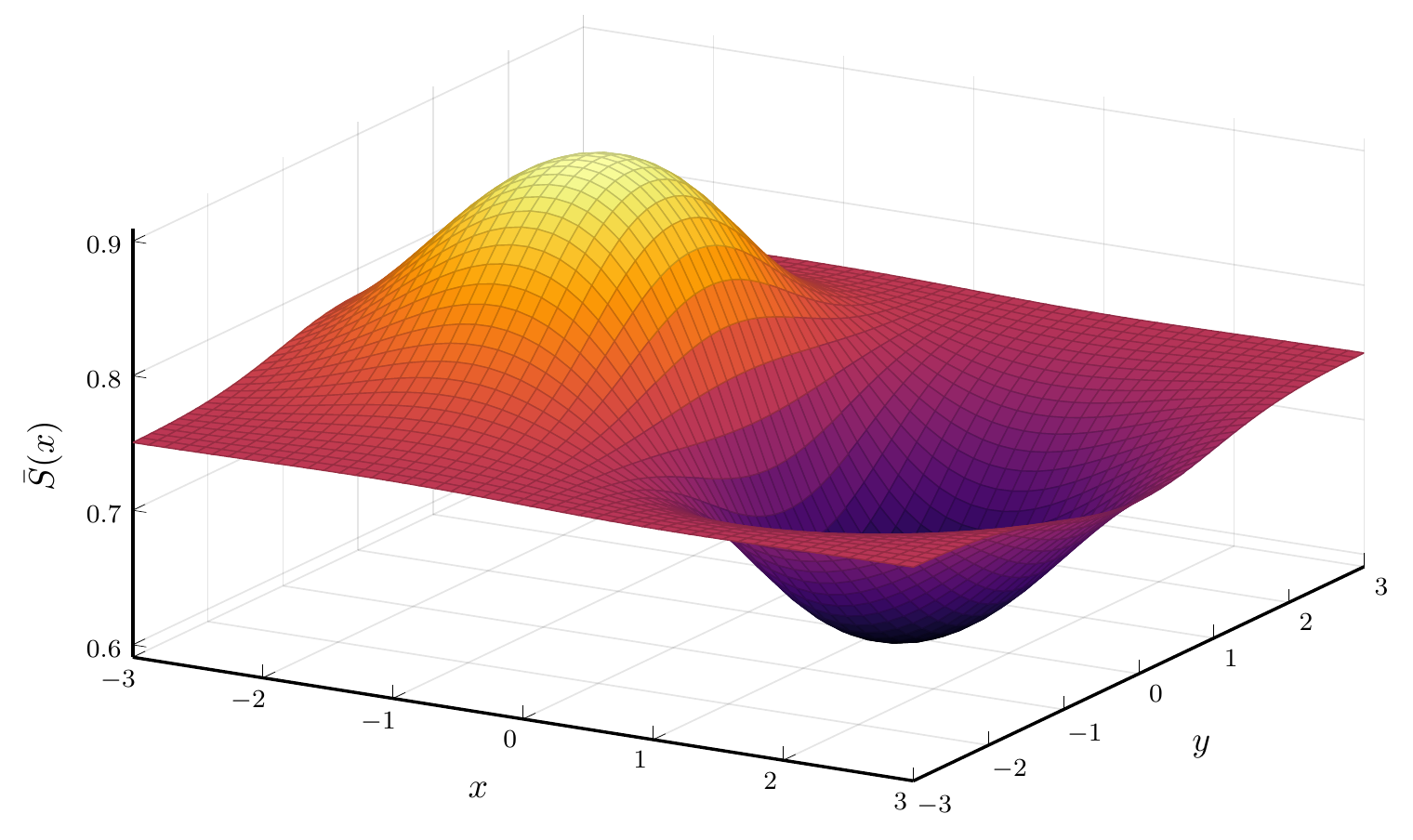}
\caption{\label{fig:org3bba911}Mean values of the parameters \(S\) according to the position of the plant}
\end{figure}

\begin{figure}[htbp]
\centering
\includegraphics[width=.9\linewidth]{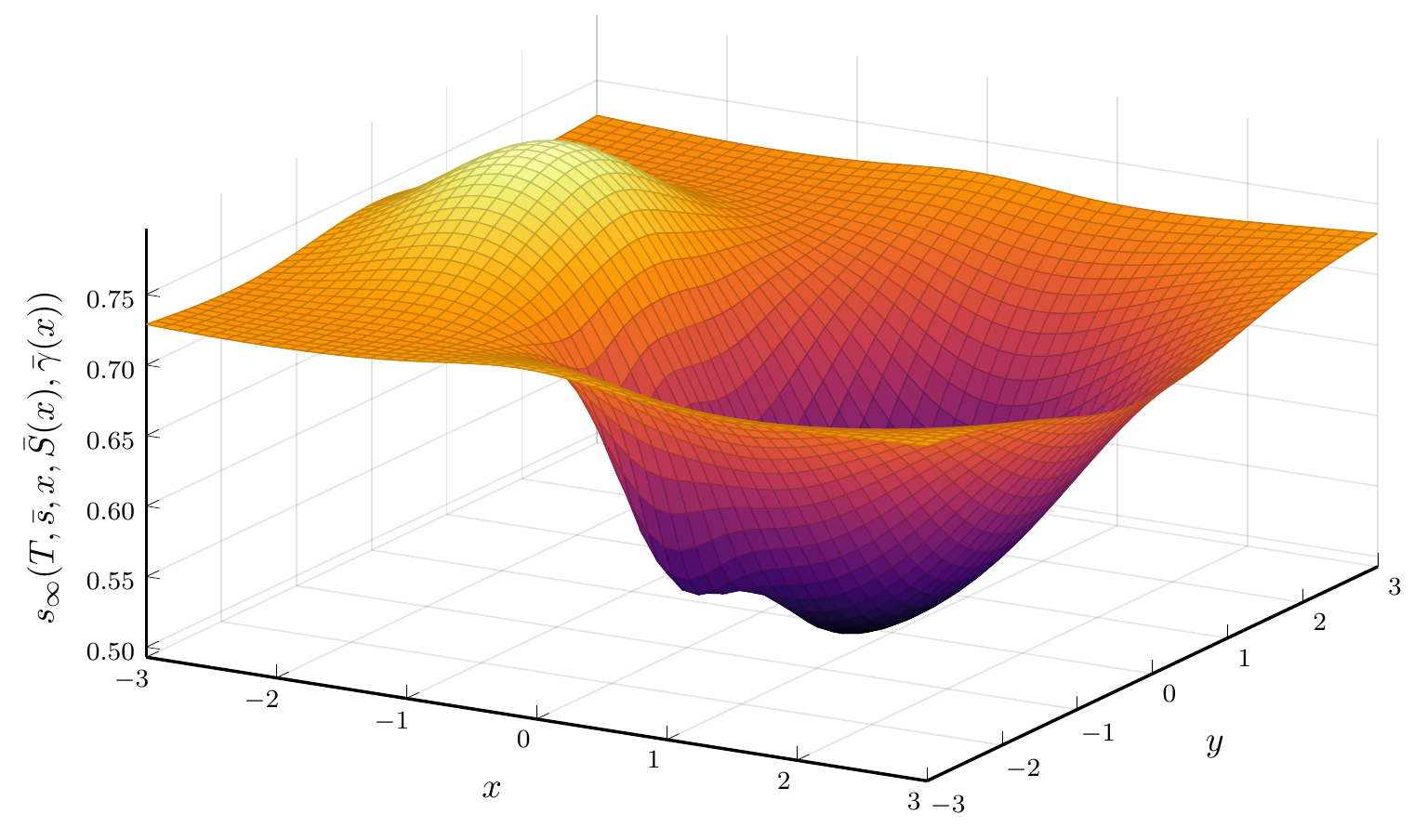}
\caption{\label{fig:org0964585}Comparison with the values of the mean-field flow at the final time \(\hat s_\infty(T,\bar s,x,\bar S(x), \bar\gamma(x))\) where \(\bar s = \displaystyle \frac{s_0^{\min}+s_0^{\max}}{2}\).}
\end{figure}

\section{Conclusion}
\label{sec:org01466d1}

The model of competition between plants studied in this article is frequently
used for its flexibility (by tuning the parameters \(\sigma_x\) and
\(\sigma_r\)) and for its capacity to reproduce the population dynamics observed
experimentally. This model has also interesting mathematical properties. Under
little restrictive assumptions on the initial conditions and parameters, the
dynamics is theoretically guaranteed to remain in a biologically interpretable
domain of the phase space. It is also possible to study the behavior of the
population when \(N\rightarrow \infty\), by simulating the mean-field dynamics, as
long as the initial distribution converges to a deterministic probability
measure. Finally, the simulations, both at the microscopic and macroscopic
scales, seem to indicate that the population tends to a stationary state when
\(t\rightarrow +\infty\). A future research direction could be to couple this
heterogeneous competition model with one of the birth and death processes, and
to study the evolution of the heterogeneity of the population.

\section{References}
\label{sec:org5a78acb}

\begin{enumerate}
\item \label{org9115032} Adams, T., Ackland, G., Marion, G., \& Edwards,
C. (2011). Effects of local interaction and dispersal on the dynamics of
size-structured populations. Ecological modelling, 222(8), 1414-1422.
\item \label{org1ddeaac} Adams, T. P., Holland, E. P., Law, R., Plank, M. J., \& Raghib,
M. (2013). On the growth of locally interacting plants: differential
equations for the dynamics of spatial moments. Ecology, 94(12), 2732-2743.
\item \label{org0f9a722} Berger, U., Piou, C., Schiffers, K., \& Grimm,
V. (2008). Competition among plants: concepts, individual-based modelling
approaches, and a proposal for a future research strategy. Perspectives in
Plant Ecology, Evolution and Systematics, 9(3-4), 121-135.
\item \label{org679f66f} Bolker, B., \& Pacala, S. W. (1997). Using moment equations to
understand stochastically driven spatial pattern formation in ecological
systems. Theoretical population biology, 52(3), 179-197.
\item \label{org759b328} Bolley, F., Canizo, J. A., \& Carrillo,
J. A. (2011). Stochastic mean-field limit: non-Lipschitz forces and
swarming. Mathematical Models and Methods in Applied Sciences, 21(11),
2179-2210.
\item \label{orgfe59c74} Beyer, R., Etard, O., Cournède, P. H., \& Laurent-Gengoux,
P. (2015). Modeling spatial competition for light in plant populations with
the porous medium equation. Journal of mathematical biology, 70(3), 533-547.
\item \label{org960cc74} Campillo, F., \& Joannides, M. (2009). A spatially explicit
Markovian individual-based model for terrestrial plant dynamics. arXiv
preprint arXiv:0904.3632.
\item \label{org56fc075} Carrillo, J. A., Goudon, T., Lafitte, P., \& Vecil,
F. (2008). Numerical schemes of diffusion asymptotics and moment closures for
kinetic equations. Journal of Scientific Computing, 36(1), 113-149.
\item \label{org3e15225} Clark, B., \& Bullock, S. (2007). Shedding light on plant
competition: modelling the influence of plant morphology on light capture
(and vice versa). Journal of theoretical biology, 244(2), 208-217.
\item \label{org179f28c} Cournède, P. H., Mathieu, A., Houllier, F., Barthélémy, D.,
\& De Reffye, P. (2008). Computing competition for light in the GREENLAB model
of plant growth: a contribution to the study of the effects of density on
resource acquisition and architectural development. Annals of Botany, 101(8),
1207-1219.
\item \label{orga691432} Dobrushin, R. L. V. (1979). Vlasov equations. Functional
Analysis and Its Applications, 13(2), 115-123.
\item \label{org9af7581} Fournier, N., \& Méléard, S. (2004). A microscopic
probabilistic description of a locally regulated population and macroscopic
approximations. The Annals of Applied Probability, 14(4), 1880-1919.
\item \label{org7f4b2af} Golse, F. (2016). On the dynamics of large particle systems in
the mean-field limit. In Macroscopic and large scale phenomena: coarse
graining, mean field limits and ergodicity (pp. 1-144). Springer, Cham.
\item \label{org8529c61} Lafitte, P., Lejon, A., \& Samaey, G. (2016). A high-order
asymptotic-preserving scheme for kinetic equations using projective
integration. SIAM Journal on Numerical Analysis, 54(1), 1-33.
\item \label{org46ed14b} Lagoutière, F., \& Vauchelet, N. (2017). Analysis and
simulation of nonlinear and nonlocal transport equations. In Innovative
Algorithms and Analysis (pp. 265-288). Springer, Cham.
\item \label{orga5cb81b} Law, R., \& Dieckmann, U. (1999). Moment approximations of
individual-based models.
\item \label{orgd4e2f62} Lv, Q., Schneider, M. K., \& Pitchford, J. W. (2008). Individualism
in plant populations: using stochastic differential equations to model
individual neighbourhood-dependent plant growth. Theoretical population
biology, 74(1), 74-83.
\item \label{orgacc2f3a} Nakagawa, Y., Yokozawa, M., \& Hara, T. (2015). Competition
among plants can lead to an increase in aggregation of smaller plants around
larger ones. Ecological Modelling, 301, 41-53.
\item \label{orgfc92cc1} Paine, C. T., Marthews, T. R., Vogt, D. R., Purves, D., Rees,
M., Hector, A., \& Turnbull, L. A. (2012). How to fit nonlinear plant growth
models and calculate growth rates: an update for ecologists. Methods in
Ecology and Evolution, 3(2), 245-256.
\item \label{org32a4d27} Rackauckas, C., \& Nie,
Q. (2017). Differentialequations. jl–a performant and feature-rich ecosystem
for solving differential equations in julia. Journal of open research
software, 5(1).
\item \label{org07ae0a9} Schneider, M. K., Law, R., \& Illian,
J. B. (2006). Quantification of neighbourhood-dependent plant growth by
Bayesian hierarchical modelling. Journal of Ecology, 310-321.
\item \label{org5f8ecba} Tsitouras, C. (2011). Runge–Kutta pairs of order 5 (4)
satisfying only the first column simplifying assumption. Computers \&
Mathematics with Applications, 62(2), 770-775.
\item \label{orgbaf95a3} Varadarajan, V. S. (1958). On the convergence of sample
probability distributions. Sankhyā: The Indian Journal of Statistics
(1933-1960), 19(1/2), 23-26.
\item \label{orgb06d6e2} Villani, C. (2009). Optimal transport: old and new
(Vol. 338, p. 23). Berlin: springer.
\item \label{orgdece5cb} Weigelt, A., \& Jolliffe, P. (2003). Indices of plant
competition. Journal of ecology, 707-720.
\end{enumerate}

\newpage
\appendix

\section{Appendix}
\label{sec:org8cc6ee7}

\subsection{Proof of inequality (\ref{eq:org34d318e})}
\label{sec:org002d1c3}
\label{org776979e}

Let us express the empirical flow and the mean-field flow (for the simplicity of notation, we omit the reference to $z^0_{1:N}$, $\hat s_N(t,s_1,\theta_1,z^0_{1:N}) = \hat s_N(t,s_1,\theta_1)$)
\begin{equation}
  \begin{aligned}
    &\hat s_N(t,s_1,\theta_1) = s_m(s_1/s_m)^{e^{-\gamma_N(\theta_1)t}}\exp
    \left(\gamma_N(\theta_1)\log(S_1/s_m)\right.\\
    &\left.\times \int_{0}^t\left(1- \frac{N}{N-1}\int_{\RR_+^*\times \RR^2}
    C(\hat s_N(\tau,s_1,\theta_1),s_1',\vert x_1-x_1'\vert )\hat \mu_N^{s,x}(\tau,d
    s_1',x_1')\right)e^{\gamma_N(\theta_1)(\tau-t)}d \tau\right)\\
    &\text{where }\gamma_N(\theta)=\gamma_N(x,S,\gamma) = \gamma\left(1-
    \frac{\log(S/s_m)}{2(N-1)R_M}\right)\\
    &s_\infty(t,s_2,\theta_2) = s_m(s_2/s_m)^{e^{-\gamma_2t}}\exp\left(\gamma_2
    \log(S_2/s_m)\phantom{\int_0^t}\right.\\
    &\times \left.\int_0^t\left(1-\int_{\RR_+^*\times \RR^2}C(s_\infty(\tau,s_2,
    \theta_2),s_2',\vert x_2-x_2'\vert )\mu^{s,x}(t,d s'_2,d x'_2)\right)
    e^{\gamma_2(\tau-t)}d \tau\right)
  \end{aligned}
\end{equation}
Let us consider the function
\begin{equation}
  \begin{aligned}
    &\mathcal S:(t,s,\gamma,S,C)\in \RR_+\times [s_0^{\min};s_0^{\max}]\times[0;\gamma_M]
    \times [S_m;s_m\exp(R_M)]\times [-1;1]\\
    &\mapsto
    s_m\left(\frac{s}{s_m}\right)^{e^{-\gamma t}}\left(\frac{S}{s_m}\right)^C
  \end{aligned}
\end{equation}
We introduce additional notations for the competition terms. Let \(s_1,s_2\in
[s_0^{\min};s_0^{\max}]\) and \(\theta_1= (x_1,S_1,\gamma_1),
\theta_2 = (x_2,S_2,\gamma_2)\in \Theta\).
\begin{equation}
  \begin{aligned}
    &C_N(t,s_1,\theta_1)\\
    &= \gamma_N(\theta_1)\int_{0}^t\left(1- \frac{N}{N-1}\int_{\RR_+^*\times
      \RR^2}C(\hat s_N(\tau,s_1,\theta_1),s_1',\vert x_1-x_1'\vert )\hat \mu_N^{s,x}(\tau,
    d s_1',x_1')\right)e^{\gamma_N(\theta_1)(\tau-t)}d \tau\\
    &C_\infty(t,s_2,\theta_2) = \gamma_2\int_0^t\left(1-\int_{\RR_+^*\times \RR^2}
    C(s_\infty(\tau,s_2,\theta_2),s_2',\vert x_2-x_2'\vert )\mu^{s,x}(t,d s'_2,d
    x'_2)\right)e^{\gamma_2(\tau-t)}d \tau
  \end{aligned}
\end{equation}
We decompose the difference between the two flows into three terms.
\begin{equation}
  \begin{aligned}
    &\hat s_N(t,s_1,\theta_1)-s_\infty(t,s_2,\theta_2) = \mathcal S(t,s_1,
    \gamma_N(\theta_1),S_1,C_N(t,s_1,\theta_1))-\mathcal S(t,s_1,\gamma_1,S_1,
    C_N(t,s_1,\theta_1))\\
    &+\mathcal S(t,s_1,\gamma_1,S_1,C_N(t,s_1,\theta_1))-\mathcal S(t,s_2,\gamma_2,
    S_2,C_N(t,s_1,\theta_1))\\
    &+\mathcal S(t,s_2,\gamma_2,S_2,C_N(t,s_1,\theta_1))-\mathcal S(t,s_2,\gamma_2,
    S_2,C_\infty(t,s_2,\theta_2))
  \end{aligned}
\end{equation}
Let us consider the first term.
\begin{equation}
  \begin{aligned}
    &\mathcal S_1=\vert \mathcal S(t,s_1,\gamma_N(\theta_1),S_1,C_N(t,s_1,\theta_1))-\mathcal
    S(t,s_1,\gamma_1,S_1,C_N(t,s_1,\theta_1))\vert \\
    &\leq \int_0^1\left\vert \frac{\partial \mathcal S}{\partial \gamma}(t,s_1,\gamma_1+
    \alpha(\gamma_N(\theta_1)-\gamma_1),S_1,C_N(t,s_1,\theta_1))d \alpha
    \right\vert .\vert \gamma_N(\theta_1)-\gamma_1\vert \\
    &\frac{\partial \mathcal S}{\partial \gamma}(t,s,\gamma,S,C) = -te^{-\gamma t}
    s_m\log(s/s_m)\left(\frac{s}{s_m}\right)^{e^{-\gamma t}}\left(\frac{S}{s_m}
    \right)^C\\
    &\mathcal S_1\leq \gamma_1ts_1\log(s_1/s_m)(S_1/S_m)\frac{\log(S_1/s_m)}{2
      (N-1)R_M}
  \end{aligned}
\end{equation}
Let us consider the second term.
\begin{equation}
  \begin{aligned}
    &\mathcal S_2 = \vert \mathcal S(t,s_1,\gamma_1,S_1,C_N(t,s_1,\theta_1))-\mathcal S(t,s_2,
    \gamma_2,S_2,C_N(t,s_1,\theta_1))\vert \\
    &\frac{\partial \mathcal S}{\partial s}(t,s,\gamma,S,C) = e^{-\gamma t}(S/
    s_m)^C(s/s_m)^{e^{-\gamma t}-1}\\
    &\left\vert \frac{\partial \mathcal S}{\partial s}(t,s,\gamma,S,C)\right\vert \leq e^{R_M}\\
    &\frac{\partial \mathcal S}{\partial S}(t,s,\gamma,S,C) = C\left(\frac{s}{s_m}
    \right)^{e^{-\gamma t}}\left(\frac{S}{s_m}\right)^{C-1}\\
    &\left\vert \frac{\partial \mathcal S}{\partial S}(t,s,\gamma,S,C)\right\vert \leq
    \frac{s_0^{\max}}{s_m}\\
    &\mathcal S_2\leq e^{R_M}\vert s_1-s_2\vert +ts^0_{\max}\log(s_0^{\max}/s_m)e^{R_M}
    \vert \gamma_1-\gamma_2\vert +(s_0^{\max}/s_m)\vert S_1-S_2\vert 
  \end{aligned}
\end{equation}
Let us consider the third term.
\begin{equation}
  \begin{aligned}
    &\frac{\partial\mathcal S}{\partial C}(t,s,\gamma,S,C) = s_m \left(
    \frac{S}{s_m}\right)^C \log \left(\frac{S}{s_m}\right)
    \left(\frac{s}{s_m}\right)^{e^{-\gamma t}}\\
    &\left\vert \frac{\partial\mathcal S}{\partial C}(t,s,\gamma,S,C)\right\vert \leq
    s_0^{\max}R_Me^{R_M}\\
    &\vert \mathcal S(t,s_2,\gamma_2,S_2,C_N(t,s_1,\theta_1))-\mathcal S(t,s_2,\gamma_2,
    S_2,C_\infty(t,s_2,\theta_2))\vert \leq \\
    &s_0^{\max}e^{R_M}R_M \vert C_N(t,s_1,\theta_1)-C_\infty(t,s_2,\theta_2)\vert 
  \end{aligned}
\end{equation}
Let us expand the difference of the competition terms:
\begin{equation}
  \begin{aligned}
    &c_N(t,s_1,\theta_1) = \int_{\RR_+^*\times \RR^2}C(\hat s_N(t,s_1,\theta_1),
    s'_1,\vert x_1-x'_1\vert )\hat \mu_N^{s,x}(t,d s'_1,d x'_1)\\
    &c_\infty(t,s_2,\theta_2) = \int_{\RR_+^*\times \RR^2}C(s_\infty(t,s_2,
    \theta_2),s'_1,\vert x_2-x'_2\vert )\mu(t,d s'_2,d x'_2)
  \end{aligned}
\end{equation}
\begin{equation}
  \begin{aligned}
    &\vert C_N(t,s_1,\theta_1)-C_\infty(t,s_2,\theta_2)\vert \\
    &\leq \left\vert \int_0^t\left(1-\frac{N}{N-1}c_N(\tau,s_1,\theta_1)\right)
    \left(\gamma_N(\theta_1)e^{\gamma_N(\theta_1)(\tau-t)}-\gamma_1
    e^{\gamma_1(\tau-t)}\right)d \tau\right\vert \\
    &+\left\vert \int_0^t\left(1-\frac{N}{N-1}c_N(\tau,s_1,\theta_1)\right)\left(
    \gamma_1e^{\gamma_1(\tau-t)}-\gamma_2e^{\gamma_2(\tau-t)}\right)d
    \tau\right\vert \\
    &+\frac{\gamma_2}{N-1}\left\vert \int_0^tc_N(\tau,s_1,\theta_1)e^{\gamma_2
      (\tau-t)}d \tau\right\vert +\gamma_2\int_0^t\left\vert c_N(\tau,s_1,\theta_1)-
    c_\infty(\tau,s_2,\theta_2)\right\vert e^{\gamma_2(\tau-t)}d \tau
  \end{aligned}
\end{equation}
Let us consider the first term.
\begin{equation}
  \begin{aligned}
    &\vert 1-\frac{N}{N-1}c_N(\tau,s_1,\theta_1)\vert \leq 1\\
    &\frac{\partial}{\partial \gamma}\left(\gamma\int_0^t\left(1-\frac{N}{N-1}
    c_N(\tau,s_1,\theta_1)\right)e^{\gamma(\tau-t)}d \tau \right)\\
    &= \int_0^t\left(1-\frac{N}{N-1}c_N(\tau,s_1,\theta_1)\right)(1+\gamma(
    \tau-t))e^{\gamma(\tau-t)}d \tau\\
    &\left\vert \frac{\partial}{\partial \gamma}\left(\gamma\int_0^t\left(1-
    \frac{N}{N-1}c_N(\tau,s_1,\theta_1)\right)e^{\gamma(\tau-t)}d \tau
    \right)\right\vert \leq \int_0^t(1+\gamma(t-\tau))e^{\gamma(\tau-t)}d \tau\\
    &= \frac{2-e^{-\gamma t}(2+\gamma t)}{\gamma}\leq t
  \end{aligned}
\end{equation}
We can deduce from the above inequality the following upper-bound on the
competition terms.
\begin{equation}
  \begin{aligned}
    &\vert C_N(t,s_1,\theta_1)-C_\infty(t,s_2,\theta_2)\vert \leq \gamma_1t\frac{\log(S_1/
      s_m)}{2R_M(N-1)}+t\vert \gamma_1-\gamma_2\vert +\frac{1}{N-1}\\
    &+\gamma_2\int_0^t\left\vert \int_{\RR^*_+\times \Theta}C(\hat s_N(\tau,s_1,
    \theta_1),s'_1,\vert x_1-x'_1\vert )\hat \mu_N^{s,x}(\tau,d s'_1,d x'_1)\right.\\
    &\left.-\int_{\RR_+^*\times \Theta}C(s_\infty(\tau,s_2,\theta_2),s'_2,\vert x_2-
    x'_2\vert )\mu(\tau,d s'_2,d x'_2)\right\vert e^{\gamma_2(\tau-t)}d \tau\\
  \end{aligned}
\end{equation}
We use the coupling \(\pi_0\) introduced at the beginning of the proof to express
the last term.
\begin{equation}
  \begin{aligned}
    &\left\vert \int_{\RR^*_+\times \Theta}C(\hat s_N(t,s_1,\theta_1),s'_1,\vert x_1-x'_1\vert )
    \hat \mu_N^{s,x}(t,d s'_1,d x'_1)\right.\\
    &\left.-\int_{\RR_+^*\times \Theta}C(
    s_\infty(t,s_2,\theta_2),s'_2,\vert x_2-x'_2\vert )\mu(t,d s'_2,d x'_2)\right\vert 
  \end{aligned}
\end{equation}
\begin{equation}
  \begin{aligned}
    \leq &\iint_{(\RR_+^*\times \Theta)^2}\vert C(\hat s_N(t,s_1,\theta_1),\hat s_N(t,s'_1,
    \theta'_1),\vert x_1-x'_1\vert )-C(s_\infty(t,s_2,\theta_2),s_\infty(t,s_2,\theta'_2),
    \vert x_2-x'_2\vert )\vert \\
    &\times \pi_0(d s'_1,d \theta_1',d s'_2,d \theta'_2)\\
  \end{aligned}
\end{equation}
We compute the derivatives of the competition potential to obtain upper-bounds
of the variations. For any \(\tilde s_1\in [s_m^N;\hat S_N(S_1)]\), any
\(\tilde s_2\in [s_m;S_2]\), and \(\delta \geq 0\), we have
\begin{equation}
  \begin{aligned}
    &\frac{\partial C}{\partial s}(\tilde s_1,\tilde s_2,\delta) = -\frac{
      \log(\tilde s_2/s_m)(1-\tanh^2(\frac{1}{\sigma_r}\log(\tilde s_2/\tilde
      s_1)))}{2R_M\tilde s_1\sigma_r(1+\delta^2/\sigma_x^2)}\\
    &\left\vert \frac{\partial C}{\partial s}(\tilde s_1,\tilde s_2,\delta)\right\vert 
    \leq \frac{1}{2s^N_m\sigma_r}\\
    &\frac{\partial C}{\partial s'}(\tilde s_1,\tilde s_2,\delta) =
    \frac{\sigma_x^2(\sigma_r(1+\tanh(\frac{1}{\sigma_r}\log(\tilde s_2/
      \tilde s_1))+\log(\tilde s_2/\tilde s_1)(1-\tanh^2(\frac{1}{\sigma_r}
      \log(\tilde s_2/\tilde s_1))))}{2R_M\tilde s_2\sigma_r(\sigma_x^2+
      \delta^2)}\\
    &\left\vert \frac{\partial C}{\partial s'}(\tilde s_1,\tilde s_2,\delta)\right\vert 
    \leq \frac{2\sigma_r+R_M}{2R_Ms_m\sigma_r}\\
    &\frac{\partial C}{\partial \delta^2}(\tilde s_1,\tilde s_2,\delta) =
    -\frac{\log(\tilde s_2/s_m)(1+\tanh(\frac{1}{\sigma_r}\log(\tilde s_2/\tilde
      s_1)))}{2R_M\sigma_x^2(1+\delta^2/\sigma_x^2)^2}\\
    &\left\vert \frac{\partial C}{\partial \delta^2}(\tilde s_1,\tilde s_2,\delta)
    \right\vert \leq \frac{1}{\sigma_x^2}
  \end{aligned}
  \label{eq:derivativeC}
\end{equation}
It follows that
\begin{equation}
  \begin{aligned}
    &\iint_{(\RR_+^*\times \Theta)^2}\vert C(\hat s_N(t,s_1,\theta_1),\hat s_N(t,s'_1,
    \theta'_1),\vert x_1-x'_1\vert )\\
    &-C(s_\infty(t,s_2,\theta_2),s_\infty(t,s_2,\theta'_2),
    \vert x_2-x'_2\vert )\vert  \pi_0(d s'_1,d \theta_1',d s'_2,d \theta'_2)\\
    &\leq
    \frac{\vert \hat s_N(t,s_1,\theta_1)-s_\infty(t,s_2,\theta_2)\vert }{2s_m^N\sigma_r}\\
    &+\frac{2\sigma_r+R_M}{2R_Ms_m \sigma_r}\iint_{(\RR_+^*\times \Theta)^2}\vert \hat
    s_N(t,s'_1,\theta_1')-s_\infty(t,s'_2,\theta'_2)\vert \pi_0(d s'_1,d
    \theta'_1,d s'_2,d \theta'_2)\\
    &+\frac{\vert x_1-x_2\vert }{\sigma_x^2}\int_{\RR^2}\vert x_1+x_2-2x'_2\vert \mu_0^x(d x'_2)\\
    &+\frac{1}{\sigma_x^2}\iint_{\RR^4}\vert x'_1+x'_2-2x_1\vert .\vert x'_1-x'_2\vert \pi_0^x(d x'_1,d x'_2)
  \end{aligned}
\end{equation}
For the last terms relative to \(x\), we have used the relation
\begin{equation}
  \vert x_1-x_1'\vert ^2-\vert x_2-x'_2\vert ^2= (x_1+x_2-2x'_2).(x_1-x_2)+(x'_1+x'_2-2x_1).(x'_1-x'_2)
\end{equation}
Let us further consider the \(x\) terms.
\begin{equation}
  \begin{aligned}
    &\int_{\RR^2}\vert x_1+x_2-2x'_2\vert \mu_0^x(d x'_2)\leq \vert x_1+x_2\vert +2\int_{\RR^2}
    \vert x'_2\vert \mu_0^x(d x'_2)\\
    &\iint_{\RR^4}\vert x'_1+x'_2-2x_1\vert .\vert x'_1-x'_2\vert \pi_0^x(d x'_1,d x'_2)\\
    &\leq \sqrt{\iint_{\RR^4}\vert x'_1+x'_2-2x_1\vert ^2\pi^x_0(d x'_1,d x'_2)}.
    \sqrt{\iint_{\RR^4}\vert x'_1-x'_2\vert ^2\pi^x_0(d x'_1,d x'_2)}\\
    &\iint_{\RR^4}\vert x'_1+x'_2-2x_1\vert ^2\pi_0^x(d x'_1,d x'_2)\\
    &\leq 2\int_{\RR^2}\vert x'_1\vert ^2\hat \mu_N^x(0,d x'_1)+2\int_{\RR^2}\vert x'_2\vert ^2
    \mu_0^x(d x'_2)-4x_1.\left(\int_{\RR^2}x'_1\hat \mu_N^x(0,d x'_1)
    +\int_{\RR^2}x'_2\mu_0^x(d x'_2)\right)\\
    &=k^x(x_1,\hat \mu_N^x(0),\mu_0^x)^2
  \end{aligned}
  \label{eq:kx}
\end{equation}
We gather all the previous inequalities to obtain an upper-bound on the gap
between the microscopic and the mean-field flows:
\begin{equation}
  \begin{aligned}
    &\vert \hat s_N(t,s_1,\theta_1)-s_\infty(t,s_2,\theta_2)\vert \leq \frac{\gamma_1ts_1
      \log(s_1/s_m)(S_1/s_m)\log(S_1/s_m)}{2R_M(N-1)}+e^{R_M}\vert s_1-s_2\vert \\
    &+ts^0_{\max}\log(s_0^{\max}/s_m)e^{R_M}\vert \gamma_1-\gamma_2\vert +(s_0^{\max}/S_m)
    \vert S_1-S_2\vert \\
    &+s_0^{\max}e^{R_M}R_M\left(\gamma_1t\frac{\log(S_1/s_m)}{2R_M(N-1)}+t
    \vert \gamma_1-\gamma_2\vert +\frac{1}{N-1}\right)\\
    &+s_0^{\max}e^{R_M}R_M\gamma_M\int_0^t\left(\frac{\vert \hat s_N(\tau,s_1,
      \theta_1)-s_\infty(\tau,s_2,\theta_2)\vert }{s_m^N\sigma_r}\right.\\
    &+\frac{2\sigma_r+R_M}{2R_Ms_m\sigma_r}\iint_{(\RR_+^*\times \Theta)^2}\vert 
    \hat s_N(\tau,s'_1,\theta'_1)-s_\infty(\tau,s_2',\theta'_2)\vert \pi_0(d s'_1,
    d \theta'_1,d s'_2,d \theta'_2)\\
    &+\frac{1}{\sigma_x^2}\left(\vert x_1+x_2\vert +\int_{\RR^2}\vert x\vert \mu_0^x(d x)
    \right)\vert x_1-x_2\vert \\
    &\left.+\frac{k^x(x_1,\hat \mu_N^x(0),\mu_0^x)}{\sigma_x^2}\iint_{\RR^4}\vert x_1'-x'_2\vert 
    \pi_0^x(d x'_1,d x'_2)\right)d \tau
  \end{aligned}
\end{equation}
Finally, we obtain the required inequality by integrating over \(\pi_0\).

\subsection{Parameters values chosen for the simulations in section \ref{org46e8607}}
\label{sec:org5e6d013}

\begin{table}[H]
\begin{tabular}{|c|c|}
\hline
initial size                                      & $s_0 = 0.1$                   \\ \hline
Extremal sizes                                                  & $s_m = 5\times 10^{-2}$, $R_M = 3$                     \\ \hline
Position variance                                               & $L = 1$                                                \\ \hline
Extremal values of $\bar S(x)$                                  & $S_m = 0.5$, $S_M = 1.0$, $S_0 = 0.75$                 \\ \hline
Positions of $\bar S(x)$ extrema                                & $x_{\max}^S = (-L,0)$, $x_{\min}^S = (L,0)$            \\ \hline
Spatial curvatures of $\bar S(x)$                               & $H_{\max}^S = H_{\min}^S = \mathrm{I}_2/L^2$           \\ \hline
Standard deviation of $S$                                       & $\delta S = 0.1$                                       \\ \hline
Extremal values of $\bar \gamma(x)$                             & $\gamma_M = 2$,$\gamma_m = 0.1$,$\gamma_0 = 1.05$      \\ \hline
Positions of $\bar \gamma(x)$ extrema                           & $x_{\min}^\gamma = (0,-L),~x_{\max}^\gamma = (0,L)$    \\ \hline
Spatial curvature of $\bar \gamma(x)$                           & $H_{\max}^\gamma = H_{\min}^\gamma = \mathrm{I}_2/L^2$ \\ \hline
Standard deviation of $\gamma$                                  & $\delta\gamma = 0.1$                                   \\ \hline
Competition parameters                                          & $\sigma_r =
1.32$, $\sigma_x = L/2$                    \\ \hline
\end{tabular}
\label{org4cb584e}
\end{table}

\subsection{Parameters values chosen for the simulations in section \ref{orgeb2d575}}
\label{sec:org3c93908}

\begin{table}[H]
\begin{tabular}{|c|c|}
\hline
Bounds of the initial size                                      & $s_0^{\min} = 0.1,~s_0^{\max} = 0.3$                   \\ \hline
Time step and time horizon                                      & $\Delta t = 1$, $T = 10$                               \\ \hline
sample size    & $N = 1~000$                                            \\ \hline
size of training and testing sets                               & $K = 1~000$                                            \\ \hline
initial degree of polynomial & $\deg(f^3_d) = 5$                                      \\ \hline
other degree  & $\deg(f^5_d) = 3$                                      \\ \hline
\end{tabular}
\label{org018e7c3}
\end{table}

\subsection{R\textsuperscript{2} quantifying the reconstruction of the competition potential of the simulation in section \ref{orgeb2d575}}
\label{sec:org1178aca}

\begin{table}[H]
\begin{tabular}{|c|c|}
\hline
Time & $R^2$ associated with the potential reconstruction \\ \hline
0    & 0.993                                          \\ \hline
1    & 0.976                                          \\ \hline
2    & 0.978                                          \\ \hline
3    & 0.980                                          \\ \hline
4    & 0.980                                          \\ \hline
5    & 0.980                                          \\ \hline
6    & 0.980                                          \\ \hline
7    & 0.981                                          \\ \hline
8    & 0.981                                          \\ \hline
9    & 0.981                                          \\ \hline
\end{tabular}
\label{orge12493f}
\end{table}

\end{document}